\documentclass[11pt]{amsart}

\usepackage[top=1in, bottom=1in, left=1.3in, right=1.3in]{geometry}
\usepackage{changes}

\usepackage{amsmath}
\usepackage{amssymb}
\usepackage{amsthm}
\usepackage{verbatim}
\usepackage{graphicx}
\usepackage{graphics}
\usepackage{color}
\usepackage{enumerate}
\usepackage{mathrsfs}
\usepackage{booktabs}
\usepackage[toc,page]{appendix}
\usepackage{enumerate}
\usepackage{mathrsfs}
\usepackage{fancyhdr}
\usepackage{amsopn, amstext, amscd,pifont}
\usepackage{amssymb,bm, cite}
\usepackage[all]{xy}
\usepackage{color}
\usepackage{graphicx,hyperref}
\usepackage{float}
\usepackage{listings}
\usepackage{setspace}
\usepackage{url}
\usepackage{paralist}
\usepackage{subcaption}
\usepackage{stmaryrd}
\usepackage{tikz-cd}
\usepackage{graphicx,pgf}
\usepackage{float}
\usepackage{enumitem}

\usepackage{fancyhdr}

\numberwithin{equation}{section} 

\newtheorem{introtheorem}{Theorem}

\newtheorem{introcorollary}[introtheorem]{Corollary}
\theoremstyle{plain}

\newtheorem{theorem}{Theorem}[section]
\newtheorem{proposition}[theorem]{Proposition}
\newtheorem{lemma}[theorem]{Lemma}
\newtheorem{corollary}[theorem]{Corollary}

\theoremstyle{definition}
\newtheorem{definition}[theorem]{Definition}

\newtheorem{remark}[theorem]{Remark}

\graphicspath{{images/}}

\def\ol{\overline}

\def\C{\mathbb{C}}

\def\H{{\mathbb{H}}}
\def\R{\mathbb{R}}

\def\Z{\mathbb{Z}}
\def\D{\mathbb{D}}
\def\CP2{{\mathbb{CP}^2}}

\def\br{\mathbf{r}}

\def\Poly{{\operatorname{Poly}}}

\def\Polyd{{\operatorname{Poly}_d}}

\def\crit{\operatorname{Crit}}
\def\post{\operatorname{Post}}

\usepackage{calrsfs}
\DeclareMathAlphabet{\pazocal}{OMS}{zplm}{m}{n}

\def\cA{{\pazocal{A}}}
\def\cB{{\pazocal{B}}}
\def\cC{{\pazocal{C}}}
\def\cD{{\pazocal{D}}}

\def\cF{{\pazocal{F}}}

\def\cJ{{\pazocal{J}}}
\def\cK{{\pazocal{K}}}

\def\cR{{\pazocal{R}}}

\def\cT{{\pazocal{T}}}
\def\cU{{\pazocal{U}}}
\def\cV{{\pazocal{V}}}

\def\field{\mathbb{C}_v}

\def\sk{\operatorname{sk}}
\def\aoneberk{\mathbb{A}^1_{an}}
\def\aone{\mathbb{A}^1}
\def\poneberk{\mathbb{P}^1_{an}}
\def\pone{\mathbb{P}^1}
\def\diam{\mathrm{diam}}

\def\dist{\mathrm{dist}_\H}

\newcommand{\iter}[2]{\ensuremath{#1^{\circ #2}}}

\usepackage[pagewise]{lineno}
\begin{document}

\title{The basin of infinity of tame polynomials}
\author{Jan Kiwi and Hongming Nie}
%\thanks{The first author was partially supported by CONICYT PIA ACT172001} 
\address{Facultad de Matem\'aticas, Pontificia Universidad Cat\'olica de Chile} 
\email{jkiwi@uc.cl}
\address{Institute for Mathematical Sciences, Stony Brook University}
  \email{hongming.nie@stonybrook.edu}
\date{\today}
\maketitle
\begin{abstract}
  Let $\C_v$ be a
  characteristic zero algebraically closed field  which is complete with respect to a non-Archimedean absolute value. We provide a necessary and sufficient condition for two tame polynomials in $\C_v[z]$ of degree $d \ge 2$ to be analytically conjugate on their basin of infinity. In the space of monic centered polynomials,
  tame polynomials with all their critical points in the basin of infinity
  form the tame shift locus. We show that a tame map $f\in\C_v[z]$ is in the closure of
  the tame shift locus if and only if  the Fatou set of $f$ coincides with the basin of infinity.
\end{abstract}

\tableofcontents
\section{Introduction}\label{sec:intro}

Let $\field$ be a characteristic $0$ algebraically closed field with residue characteristic $p\ge 0$
which is complete with respect to a non-Archimedean absolute value $|\cdot|$, that is assumed to be non-trivial. 
%Consider the affine line $\aone$  over $\field$ as well as
%the corresponding Berkovich analytic line $\aoneberk$.
We regard nonconstant polynomials $f \in \field[z]$
as dynamical systems acting on the Berkovich analytic line $\aoneberk$ over $\field$.
The dynamics of a degree $d \ge 2$ polynomial $f$  partitions $\aoneberk$ into two sets: the \emph{basin of infinity} $\cB(f)$
consisting of all
$x \in \aoneberk$ with unbounded orbit and, its complement, the \emph{filled Julia set} $\cK(f)$, which is compact. 
The theme of this paper is to investigate when the actions of two polynomials  on their basins of infinity are analytically conjugate. However, in this introduction we start discussing some consequences of our investigation and finish
describing directly related results.

%%The reader may refer to~\cite{Benedetto19} for a general background on dynamics over non-Archimedean fields.
%%Throughout this  paper, we fix $d\ge 2$.

It is well known that the dynamics of a polynomial $f$ is strongly influenced by the dynamical behavior of points $x$ where $f$ fails to be locally injective.
The set of such points  $x \in \aoneberk$  is called 
the \emph{ramification locus} $\cR(f) \subset \aoneberk$ (c.f.~\cite{Faber13II}).
We will focus on the more tractable class of \emph{tame polynomials}
consisting on maps $f$ for which $\cR(f)$ is a locally finite tree (c.f.~\cite{Trucco14}).
We say that a tame polynomial lies in the \emph{tame shift locus} if $\cR(f)$ is contained
in the basin of infinity, equivalently, if the critical points of $f$ are
contained in the basin of infinity. In this case, the unique Fatou component
of $f$ is the basin of infinity and the Julia set
$\cJ(f) := \partial \cK(f)$ is formed by type I points. 
That is, $\cJ(f)$ is contained in the (classical) affine line $\aone$.
Moreover, 
$f: \cJ(f) \to \cJ(f)$ is topologically conjugate to the one-sided shift on $d$ symbols, (c.f.~\cite[Theorem 3.1]{Kiwi06}).
%(e.g. see~\cite[Section 3]{Kiwi06}).
 %\footnote{hyperbolic in the sense of benedetto} 

In parameter space, for simplicity,
we work in the space $\Poly_d$ of \emph{monic and centered  polynomials  of degree $d$} which, via coefficients, is naturally identified with $\field^{d-1}$.
%THIS SHOULD GO ELSEWHERE
% Namely, the elements of $\Poly^{cm}_d$ are $(f, c_1, \dots, c_{d-1})$
% where $f$ is a degree $d$ polynomial and $(c_1,\dots, c_{d-1}) \in \field^{d-1}$
% are such that $c_1 + \cdots + c_{d-1} =0$ and $f'(z) =0$ if and only if $z = c_i$ for some $i$.
% Via $(f,c_1\dots,c_{d-1}) \mapsto (f(0),c_1\dots,c_{d-1})$,
Unless otherwise stated,  we work with the metric topology in $\Poly_d$
induced by the (sup) norm on $\field^{d-1}$.
%With some abuse of notation we
%refer to elements of  $\Poly^{cm}_d$ as polynomials $f$ with (marked) critical points $c_1(f), \dots, c_{d-1}(f)$.
%
% The aim of this paper is to study sufficient conditions for two polynomials
% to be analytically conjugate in their basins of infinity. Since compact subsets of the classical line $\aone$ are analytically removable, we derive rigidity results for (tame) polynomials
% whose Julia set is entirely contained in $\aone$. 
% We will be interested on polynomials whose Julia set is contained in the classical line $\aone$. 
%
% The filled Julia set $\cK(f)$ of such a polynomial $f$ is formed by all
% $x \in \aoneberk$ with bounded orbit (i.e. $\iter{f}{n} (x) \not\to \infty$ as $n \to \infty$). The Julia set $\cJ(f) = \partial \cK(f)$. 
%
% The critical points of $f$ are the classical points $z \in \aone$ where the map is not locally injective.  Equivalently,
% the characteristic of the residue field 
%  the characteristic of $\rfield$ does not divide
% the local degree of $f$ at $x$  for all $x \in \cR(f)$.
 Our investigation regarding the basin of infinity will allow us to characterize
the location in parameter space of tame polynomials whose basin of infinity agree with the Fatou set or, equivalently, the Julia set is contained in the classical affine line $\aone$.

% For simplicity we work in the space $\Poly_d^{cm}$
% of critically marked monic
% centered polynomial of degree
% $d$. That is,  elements of $\Poly_d^{cm}$ are of the form $(f,c_1,\dots,c_{d-1})$
% where $f $ is a monic polynomial of degree $d$ such that 
% its critical points are $\{ c_1, \dots, c_{d-1}\}$ and $\sum c_i =0$.
% Thus, $\Poly_d^{cm}$ is naturally identified with a hyperplane in $\field^d$.
% We consider the metric on $\Poly_d^{cm}$ induced by the max-norm on  $\field^d$ and by default open sets in $\Poly_d^{cm}$ will be with respect to the metric topology. 

\begin{introtheorem}
  \label{ithr:A}
  Consider a tame  polynomial $f \in \Poly_d$.
  Then  $\cJ (f) \subset \aone$ if and only if $f$ is in the closure of the tame shift locus. 
\end{introtheorem}
%\footnote{is it an if and only if?, for fields such that trucco holds?, do we characterize all tame maps with all cycles repelling}
Note that each
polynomial in the tame shift locus is expanding on its Julia set (c.f~\cite{Benedetto01,Lee19}). Hence we show that tame polynomials with $\cJ (f) \subset \aone$
can be perturbed to maps which, in a certain sense, have hyperbolic (expanding)
dynamics over their Julia sets.

When a tame polynomial $f$ is in the closure of the tame shift locus, it is not difficult to deduce from Trucco's work~\cite{Trucco14} that $\cJ (f) \subset \aone$. Thus, our contribution is to show that Julia critical points (if any) of
a tame polynomial $f$ with $\cJ (f) \subset \aone$ simultaneaously
become escaping under an appropriate arbitrarily small perturbation of $f$.
We say that a Julia critical point $c \in \cJ(f)$ is \emph{active at $f$} 
if there exists a polynomial $g$ arbitrarily close to $f$ with a critical point $c'$ close to $c$ such that $c'$ lies
in the basin of infinity under iterations of $g$.
% In general, we do not know if
% the converse of the theorem holds.
% However, this follows from Trucco's non-wandering Theorem~\cite[Theorem~A]{Trucco14} when  $\field$ contains a discrete valuation subfield $L$ such that
% algebraic elements over $L$ are dense in $\field$.
% \medskip
% 
% To prove the converse we show that all the Julia critical points
% of a tame polynomial $f$ with $\cJ(f) \cap \aone$ are \emph{active}.
% Namely, given a critical point in $c \in \cJ(f)$ there exists a $g$ arbitrarily close to $f$ with a critical point $c'$ close to $c$ such that $c'$ lies
% in the basin of infinity under iterations of $g$. 
% We will show that the converse is the fact that compact subsets of $\aone$ are analytically removable. 
% builds on the notions of \emph{active} and \emph{passive critical points}.
% Although we will discuss these definitions in greater detail in Section, let us say here that
% a critical point $c_i(f) \in \cJ(f)$ is active at $f \in \Poly_d^{cm}$
% if there are arbitrarily close  polynomials for which this critical point lies in the basin of infinity.
(We must warn the reader that our definition of active/passive
critical points differs from the one by Irokawa~\cite{Irokawa19}.)

\begin{introcorollary}
  \label{icor:B}
  Assume that $f \in \Poly_d$ is a 
  tame polynomial such that $\cJ(f) \subset \aone$.
  If a critical point $c$ lies in the Julia set $\cJ(f)$,
  then $c$ is active at $f$.
\end{introcorollary}

Without any assumption on $f \in \Poly_d$,
we conjecture that if a critical point $c$ lies in 
$\cJ(f)$, then $c$ is active at $f$.
%A similar statement should hold for rational maps.
That is, we conjecture that bifurcations must occur at maps with Julia critical points. On the other hand,
$J$-stability and ``hyperbolicity'' results for maps whose
Julia set is critical point free have been obtained by  Benedetto and Lee~\cite{Benedetto01,Lee19, Benedetto22}, and T. Silverman~\cite{Silverman17,Silverman19}.
Although useful, analogies with open problems and results in complex dynamics should be taken with extreme caution. In fact, the hypothesis that $c$ is a Julia critical point is \emph{stronger}
in the non-Archimedean context than in the complex setting since it immediately implies a control on the geometry around $c$ (i.e. there exists a sequence of nested Fatou annuli around $c$ with divergent sum of moduli). 
So although 
the ``analogue'' conjecture in complex dynamics has proven to be hard and elusive, even for quadratic polynomials,
it might be the case that establishing activity for
Julia critical points is more accessible
in the non-Archimedean setting.
% But the consequences
%  at the moment rather limited since we lack
% of a dynamical stratification of non-Archimedean moduli spaces
% analogous to the one provided by~\cite{MSS,L} in the complex setting. 

% In fact, disconnected Julia sets with no periodic critical component may be regarded as the analogue of polynomials in trucco type fields with julia contained in $\aone$.
% Only recently (reference?) it was stablished that this maps are in the closure of the shift loci? as we show here establishing this fact in the non-Archimedean setting is far more simple mainly because any compact subset of $\aone$ is analytically removable. 
% Analytic conjugacies between basins of infinity automatically extend to the Julia set. 

% which is the unique germ of an analytic map
% at $\infty$ such that
% $\phi_f (f(z)) = \phi_f(z)^d$ and
%  $\phi(z)/z \to 1$ as
%  $\aone \ni z \to \infty$.

\medskip
To study perturbations of a given polynomial it is convenient to consider analytic families of critically marked monic and centered polynomials 
$\{(f_\lambda, c_1(\lambda), \dots, c_{d-1}(\lambda)) \}$ parametrized by a disk $\Lambda \subset \field$ (see \S~\ref{sec:analytic-family}). In contrast with T. Silverman's work, for our purpose, we
only consider parameters $\lambda \in \field$
and polynomials with coefficients in $\field$, that is we do not work with
(non-classical)  parameters in the Berkovich disk associated to $\Lambda$
(c.f.~\cite{Silverman17,Silverman19}). 
A critical point $c_i(\lambda)$ is called \emph{passive in $\Lambda$} if either $c_i(\lambda) \in {\cB(\lambda)}:=\cB(f_\lambda)$ for all $\lambda \in \Lambda$ or $c_i(\lambda) \notin \cB(\lambda)$ for all $\lambda \in \Lambda$.
Analytic conjugacies between actions on basins of infinity must respect
critical orbits and agree, up to normalization, with a conjugacy
furnished near infinity by the B\"ottcher coordinates.
Thus, it is convenient to prescribe the locations of the critical points in
$\cB(\lambda)$  with the aid of the
B\"ottcher coordinate $\phi_\lambda$ which conjugates $f_\lambda$ with $z \mapsto z^d$ in a neighborhood of $\infty$ (see \S~\ref{sec:poly}).
We say that the \emph{B\"ottcher coordinate of a critical point $c_i(\lambda) \in \cB(\lambda)$
 is constant in $\Lambda$} if $\phi_\lambda(\iter{f_\lambda}{n}(c_i(\lambda)))$ is a constant function of $\lambda \in \Lambda$ for some $n$ sufficiently large.
Along analytic families with passive critical points having constant B\"ottcher coordinates, the analytic dynamics in the basin of infinity is constant:

%We say an analytic family in $\Poly_d^{cm}$ is \emph{passive} if all critical points are passive.
\begin{introtheorem}
  \label{ithr:C}
  Consider an analytic family  $\{(f_\lambda, c_1(\lambda), \dots, c_{d-1}(\lambda)) \}$ of critically marked tame monic and centered polynomials 
  parametrized by an open disk $ \Lambda \subset \field$.
  Assume that all critical points are passive and that the B\"ottcher coordinates of  escaping critical points are constant in $\Lambda$. 
 Then for all $\lambda_1, \lambda_2 \in \Lambda$,
the maps  $f_{\lambda_1}:\cB(\lambda_1) \to \cB(\lambda_1)$ and $f_{\lambda_2}:\cB(\lambda_2) \to \cB(\lambda_2)$ are analytically conjugate.
  Moreover, if in addition, $\cJ(f_{\lambda_0}) \subset \aone$ for some $\lambda_0\in\Lambda$, then $f_\lambda = f_{\lambda_0}$ for all $\lambda \in \Lambda$.
\end{introtheorem}

The moreover part of the statement above
is a  consequence of
the  fact from analytic geometry that
compact subsets of $\field$ are analytically removable
(see Theorem~\ref{thr:removability} or \cite[Proposition~2.7.13]{Fresnel04}).\footnote{
This fact is 
analogous to the removability of
absolute measure $0$ subsets of the complex plane, e.g. see~\cite{McMullen94}.}
In complex polynomial dynamics DeMarco and
Pilgrim~\cite{DeMarco11} studied
the map which assigns to each element of moduli space its dynamics
on the basin of infinity (modulo analytic conjugacy). In this language, removability implies that, for maps in the closure of the tame shift locus,
the analytic dynamics on their basins of infinity determines the map uniquely (modulo affine conjugacy).
For a general $f_0 \in \Poly_d$ it would be interesting to describe the locus of maps $f \in \Poly_d$ whose action on $\cB(f)$ is analytically conjugate to
$\cB(f_0)$ (c.f.~\cite{DeMarco11}). 

\medskip
Given a tame polynomial $f$, we define the \emph{dynamical core} of $f$
as the smallest forward invariant
set $\cA_f$ containing the non-classical ramification
points in $\cB(f)$.
It is not difficult to show that
$\cA_f$ is a locally finite tree.
Then, given any pair of tame polynomials $f$ and $g$ in $\Poly_d$
we introduce  the notion of
an \emph{extendable conjugacy} $h: \cA_f \to \cA_g$.
Loosely speaking, an extendable conjugacy is a diameter preserving isometric conjugacy
that is locally a translation and agrees with a B\"ottcher coordinate change near infinity.
 For a precise definition see
 \S~\ref{sec:conjugacy}.
 Analytic geometry will imply that, modulo a root of unity, an
 analytic conjugacy
 $\varphi: \cB(f) \to \cB(g)$ restricts to an 
 extendable conjugacy $h: \cA_f \to \cA_g$ and, conversely, 
via an analytic continuation argument we prove that 
extendable conjugacies upgrade to analytic conjugacies:

% The result from which the previous theorems are deduced 
% Motivate GLOBAL problem of determining conjugacy in basin of infinity \cite{DeMarco17}.
% Given a tame polynomial $f$ we consider the smallest forward invariant
% set $\cA_f$ containing $\cR(f) \cap \cB(f) \cap \H$.
% We will show that $\cA_f$ is a locally finite tree.
% Given any pair of polynomials $f$ and $g$ in $\polycmd$, via the B\"ottcher
% coordinate change $\phi_g^{-1} \circ \phi_f$, we have that $f$ and $g$
% are analytically conjugate near infinity.

\begin{introtheorem}
  \label{ithr:D}
  Let $f, g \in \Poly_d$ be tame polynomials.
  There exists an extendable conjugacy  between  $f: \cA_f \to \cA_f$ and $g: \cA_g \to \cA_g$ if and only if there exists an  analytic conjugacy 
  between $f: \cB (f) \to \cB(f)$ and  $g: \cB (g) \to \cB(g)$
  which %is asymptotic to the identity. 
  extends the B\"ottcher coordinate change.
\end{introtheorem}

%IMPORTANT: IS THIS AN IF AND ONLY IF. THAT IS ARE ALL TANGENT MAPS OF PHI TRANSLATIONS? YES IF WE SAY THAT VARPHI IS ASYMPTOTIC TO IDENTITY AT INFINITY

Let us now give an overview of the paper.

Section~\ref{sec:preliminaries} is devoted to preliminaries.
We start summarizing basic facts and notation related to the Berkovich affine line $\aoneberk$ in \S~\ref{sec:berkovich-space}.
The space of monic and centered polynomials $\Poly_d$ is introduced in \S~\ref{sec:poly}. Here we also discuss basic (dynamical) objects associated to a given polynomial such as the \emph{base point} and \emph{B\"ottcher coordinates}, and introduce the \emph{dynamical core}. Analytic families of polynomials
is the topic discussed in \S~\ref{sec:analytic-family} with emphasis on definitions such as passive/active critical points and constant B\"ottcher coordinates.

Section~\ref{sec:extendable} is devoted to the proof of Theorem~\ref{ithr:D}.
Extendable conjugacies are defined in~\S~\ref{sec:conjugacy}.
In \S~\ref{sec:analytic2extendable}, one direction of Theorem~\ref{ithr:D}
is established, namely that analytic conjugacies restrict to extendable conjugacies between dynamical cores (Corollary~\ref{c:ithr-D-easy}). The other direction
is obtained via an ``analytic continuation'' argument in \S~\ref{s:maps-annuli}
which employs a lemma proven in Appendix~\ref{appendix}.

In Section~\ref{sec:existence} we prove a slightly stronger version of Theorem~\ref{ithr:C} which relies on the notion of $\rho$-close B\"ottcher coordinates.
The stronger version of Theorem~\ref{ithr:C} will be needed to establish Theorem~\ref{ithr:A}.
Intuitively, given two polynomials $f, g$ we introduce $\rho \in ]0,\infty]$ to quantify
how close are the B\"ottcher coordinates of escaping critical points.
With this notion,
$\rho=\infty$ corresponds to having the same B\"ottcher coordinates.
Most of the proof is in \S~\ref{sec:passive2extendable}
after we introduce the $\rho$-trimmed dynamical core in~\S~\ref{sec:trimmed}
and use it to define the notion of $\rho$-close B\"ottcher coordinates in~\S~\ref{sec:close}.

Section~\ref{sec:rigidity} is devoted to prove Theorem~\ref{ithr:A} and
Corollary~\ref{icor:B}. In \S~\ref{sec:bounded} we deduce from Trucco's work
\cite{Trucco14} that tame polynomials in the closure
of the tame shift locus do not have bounded Fatou components.
In \S~\ref{sec:perturbation}, we finish the proof of the aforementioned results
constructing suitable analytic families to perturb into the shift locus a tame polynomial $f$ such that $\cJ(f) \subset \aone$. 

{\subsection*{Acknowledgements}
The authors would like to thank Charles Favre for useful conversations. The first author was partially supported by CONICYT PIA ACT172001.

\section{Preliminaries}
\label{sec:preliminaries}
For future reference  we
introduce basic definitions, notation and results
regarding Berkovich space {and analytic maps} in \S~\ref{sec:berkovich-space},
regarding the dynamics of monic and centered polynomials in \S~\ref{sec:poly}
and, regarding analytic families of such polynomials in \S~\ref{sec:analytic-family}.

\subsection{Berkovich space {and analytic maps}}
\label{sec:berkovich-space}
The elements of $\aoneberk$ are the multiplicative seminorms in $\field [z]$
that extend the absolute value $| \cdot |$ on $\field$.
As customary we write elements of $\aoneberk$ simply by $x$
and, when necessary, the corresponding seminorm of
$f(z) \in \field [z]$ is denoted  by $\| f(z) \|_x$. Also, to ease notations, let
 $|x| := \|z\|_x$. 
Via the identification of $ x\in \aone$ with the seminorm
$\|f(z)\|_x = |f(x)|$, we may regard $\aone$ as a subset of $\aoneberk$. 
Most of our work occurs in $\aoneberk$, however sometimes it is convenient to
regard $\aoneberk$ as naturally embedded in the Berkovich projective line
$\poneberk$ which (set theoretically) is obtained by adding one point, denoted by $\infty$, to $\aoneberk$.
The default topology for $\aoneberk$ and $\poneberk$
will be the Gel'fand topology.
With this topology $\aoneberk$ is locally compact and $\poneberk$ is the one-point compactification of $\aoneberk$.
Moreover, $\aone$ and $\pone$ are dense in $\aoneberk$ and $\poneberk$, respectively.
We refer the reader
to~\cite{Baker10, Benedetto19} for general background on the Berkovich affine and projective line specially adapted to one dimensional non-Archimedean dynamics.

%Given $a \in \field$ and $r >0$, 
%\begin{eqnarray*}
%  D(a,r) & : = & \{ x \in \aoneberk : \| z -a \|_x < r \}, \\
%  \ol{D}(a,r) & : = & \{ x \in \aoneberk : \| z -a \|_x \le r \}
%\end{eqnarray*}
%are \emph{Berkovich open and closed disks}, respectively.

%\jk{Introduce here Partial order $\prec$ in $\aoneberk$}
Given $x \in \aoneberk$, let
$$\ol{D}_x := \left\{ y \in \aoneberk : \|f(z)\|_y \le  \|f(z)\|_x, \,\,  \forall f(z) \in \field[z] \right\}.$$
According to Berkovich's classification of seminorms,
$\ol{D}_x \cap \aone$  is
 a (possibly degenerate) $\C_v$-closed disk of radius $r \ge 0$ or the empty set.
If $r=0$ (resp. $r \in |\field^\times|$, $r \notin |\field^\times|$), then
$x$ is called a type I (resp. II, III) point. When $\ol{D}_x \cap \aone$
is empty, $x$ is called a type IV point.
The \textit{Gauss point}, denoted by $x_G$,
is the unique point whose associated disk
$\ol{D}_{x_G}$ is $\{ z \in \field : |z| \le 1 \}$.
It follows that $x_G$ is of type II.

Any point in $\aoneberk$ can be represented by the cofinal equivalence class of a decreasing sequence of $\C_v$-closed disks, so $\aoneberk$ is equipped with a partial order $\prec$ that extends the inclusion relation on the $\C_v$-closed disks, in particular, for $x_1,x_2\in\aoneberk$ of type I, II or III, $x_1\prec x_2$ if and only if $\ol{D}_{x_1} \cap \aone\subset\ol{D}_{x_2} \cap \aone$. The partial order $\prec$ extends to $\poneberk$ by declaring $\infty$ to be the unique maximal element in $\poneberk$. This partial order $\prec$ endows $\poneberk$ and $\aoneberk$ with a tree structure. Given two distinct points $y_1,y_2\in\poneberk$, denote by $[y_1,y_2]$ (resp. $]y_1,y_2[$) the closed (resp. open) segment in $\poneberk$ connecting $y_1$ and $y_2$. 

For $a \in \field$ and $r >0$, the \emph{Berkovich open and closed disks} with radii $r$ containing $a$, respectively are
\begin{eqnarray*}
  D(a,r) & : = & \{ x \in \aoneberk : \| z -a \|_x < r \}, \\
  \ol{D}(a,r) & : = & \{ x \in \aoneberk : \| z -a \|_x \le r \}.
\end{eqnarray*}
%are \emph{Berkovich open and closed disks}, respectively.
For any $x\in\aoneberk$, after defining the \emph{diameter} of $x$ by
$$\diam(x) := \inf \{ \| z - a \|_x : a \in \field \},$$
 we have $\ol{D}_x= \ol{D}(a,\diam(x))$ for all $a\in\ol{D}_x \cap \aone$ if $x$ is of type I, II, or III, and therefore $\diam(x)$  coincides with the diameter $r$ of $\ol{D}_x \cap \aone$; in this case we write $x=x_{a,r}$.  If $x$ is of type IV, then $\diam(x)>0$, while $\ol{D}_x \cap \aone$ is empty.

 The space $\H:=\aoneberk \setminus \aone$ is equipped with a metric $\dist$ called \emph{the hyperbolic metric} which is invariant under affine
 transformations. With respect to this metric $\H$ is an
$\R$-tree. The subspace topology in $\H$ is strictly coarser than the metric topology. We normalize $\dist$ so that $$\dist(x_{0,r}, x_{0,s}) = \log (s/r)$$
for all $0<r<s$.

We say that a closed and connected
subset $\aoneberk$ or $\poneberk$ is a 
\emph{subtree}. Unless explicitly stated, a  subtree is equipped with neither vertices nor edges.  However, for a subtree $\Gamma\subset\aoneberk$,  the \emph{valence} of $\Gamma$ at $x\in\Gamma$ is the number (maybe $\infty$) of connected components of $\Gamma\setminus\{x\}$. 
We denote by $\partial\Gamma$ the set of points in $\Gamma$ having valence $1$. A nonempty connected subset $S\subset\aoneberk$ is an \emph{open subtree} if $S=\ol{S}\setminus\partial\ol{S}$; so in this case $S$ is contained in $\mathbb{H}$ and any point in $S$ has valence at least $2$, moreover, for simplicity, set $\partial S:=\partial\ol{S}$. An (open) subtree in $\aoneberk$ is \emph{locally finite} if each point has finite valence and a neighborhood containing finitely many branch points of this (open) subtree. Given an open and connected set $U \subsetneq \aoneberk$, the \emph{skeleton} of $U$, denoted by $\sk U$, is the set formed by all $x \in U$ %such thaat least one direction at $x$ distinct from the direction of infinity contains  points of $\aoneberk  \setminus U$. In other words, skeleton points 
that separate the complement of $U$ in $\poneberk$.
It follows that $\sk U$ is a locally finite open subtree in $\mathbb{H}$. %of $\H$.
%That is, $\sk U$ is connected and; for all $x \in \sk U$, we have that $\sk U \setminus \{x \}$ has at least $2$ and at most finitely many connected components and, $x$ has a neighborhood with finitely many branch points. 

 Given a point $x \in \aoneberk$, a connected component of $\aoneberk \setminus \{x\}$ corresponds to a \emph{direction $\vec{v}$ at $x$}.
Denote by $D_x(\vec{v})$ the component of $\aoneberk \setminus \{x\}$
corresponding to the direction $\vec{v}$ at $x$.
When the point $x$ is clear from context or irrelevant, we will sometimes write
$D(\vec{v})$ for $D_x(\vec{v})$.
The set of all directions at $x$ forms the \emph{space of directions or (projectivized) tangent space at $x$}, denoted by $T_x \aoneberk$.
It will also be convenient to denote the unique direction at $x$ 
`containing' $z\in\aoneberk$  by $\vec{v}_x(z)$ and denote by $\vec{v}_x(\infty)$ the unique direction at $x$ such that $D(\vec{v}_x(\infty))$ is unbounded.
At type I and IV points, the tangent space is a singleton and
at type III, consists of two directions.
However, at type II points the tangent space is naturally
endowed with the structure of a projective line
over the residue field of $\field$ with a distinguished point $\infty$.

Given an open set $U \subset \aoneberk$, an \emph{analytic map} $\psi:U \to \aoneberk$ is a morphism between the corresponding analytic structures.
Analytic maps can be described in a more concrete manner as follows.
An \emph{affinoid} $X$ in $\aone$ is a subset of the form $\overline{B}(a,r) \setminus \cup B(a_i,r_i)$ where $r, r_i \in |\field^\times|$ and $a, a_i \in \aone$. The corresponding Banach $\field$-algebra $\cA(X)$ is formed by
the uniform limits of rational maps with poles outside $X$ endowed with the sup-norm $\|\cdot\|_X$ on $X$. Then $X_{an}$ is the space of all multiplicative seminorms on
$\cA(X)$ bounded by  $\|\cdot\|_X$ that agree with $|\cdot|$ on the constants, endowed with the Gel'fand topology.
Restriction to $\field[z]$ identifies $X_{an}$
with $\overline{D}(a,r) \setminus \cup D(a_i,r_i) \subset \aoneberk$. Analytic maps $\psi: X_{an} \to \aoneberk$ are naturally identified with $\cA(X)$. That is, given $\varphi \in \cA(X)$, $x \in X_{an}$ and $q(z) \in \field[z]$, then $\|q(z)\|_{\varphi(x)} = \|q(\varphi(z))\|_x$ defines a morphism $\varphi:  X_{an} \to \aoneberk$.
It follows that for any given arbitrarily large
hyperbolic ball in $\H$, the action of $\varphi:  X_{an} \to \aoneberk$ agrees with that of a rational map with poles outside $X$.
Given an open set $U \subset \aoneberk$, we have that $\psi: U \to \aoneberk$
is analytic if its restriction to every affinoid $X$ contained in $U$ is analytic. 
Assuming that $U$ is connected, nonconstant analytic maps
$\psi: U \to \aoneberk$ are open maps with isolated zeros.
Therefore, analytic maps that coincide on an open subset of $U$ are equal.
Furthermore, we have the following analytic removability result (c.f \cite[Proposition~2.7.13]{Fresnel04}):

\begin{theorem}\label{thr:removability}
  Assume that $K \subset \aone$ is a compact set contained in a disk $D \subset \aoneberk$.
  If $f: D \setminus K \to \aoneberk$ is a bounded analytic map, then
  $f$ extends to a analytic map $\bar{f}: D \to \aoneberk$. 
\end{theorem}

Assume that $U$ is an open set containing  $x \in \aoneberk$ and
$\psi: U \to \aoneberk$ is a non-constant analytic map. Then $\psi$ has a well defined local degree at every $x \in U$ which we denote by $\deg_x \psi$.
Write $y:= \psi(x)$, and define 
 %and denote 
 the \emph{tangent map} 
 $T_x\psi: T_x \aoneberk \to T_y \aoneberk$ of $\psi$ at $x$ 
%The tangent map is  
 by the property that there exists a neighborhood
 $V$ of $x$ such that $\psi(D_x(\vec{v}) \cap V) \subset D_{y} (T_x\psi(\vec{v}))$ for any $\vec{v}\in T_x \aoneberk$. The direction $T_x\psi(\vec{v})\in T_y\aoneberk$ exists and it is independent of the choice of $V$. When $x$ is a type II point, the map $T_x\psi$ is a non-constant rational map over the residue field of $\field$. Moreover, the local degree $\deg_x\psi$ of $\psi$ at $x$ coincides with the degree of  $T_x\psi$.

 {An analytic map is piecewise linear with respect to the hyperbolic metric.
The scaling factors or ``slopes'' are determined by the local degrees (c.f. \cite[Theorem 9.35]{Baker10}):}

\begin{lemma}\label{lem:factor}
 Let $\psi$ be a non-constant analytic map on an open set $U\subset \aoneberk$. If $\psi$  has constant local degree on a segment $]x_1,x_2[\subset U$, then  for any $x\in]x_1,x_2[$
 $$\dist(\psi(x_1),\psi(x_2))=\deg_x\psi\cdot\dist(x_1,x_2).$$
 \end{lemma}

%The space $\H:=\aoneberk \setminus \aone$ is equipped with a metric $d_\H$ called \emph{the hyperbolic metric}. With respect to this metric $\H$ is an $\R$-tree. The subspace topology in $\H$ is strictly coarser than the metric topology. 

%Given an open and connected set $U \subsetneq \aoneberk$, the \emph{skeleton of $U$}, denoted $\sk U$, is the set formed by all $x \in U$ such that
%at least one direction at $x$ distinct from the direction of infinity 
%contains  points of $\aoneberk  \setminus U$. In other words, skeleton points separate the complement of $U$ in $\aoneberk \cup \{\infty\}$.
%It follows that $\sk U$ is a \emph{locally finite open subtree of $\H$}.
%That is, $\sk U$ is connected and; for all $x \in \sk U$, we have that $\sk U \setminus \{x \}$ has at least $2$ and at most finitely many connected components and, $x$ has a neighborhood with finitely many branch points. 

\subsection{Polynomial dynamics}\label{sec:poly}

% \subsubsection{The space $\polycmd$}\label{sec:poly}
We denote by $\Poly_d$ the space of \emph{monic and centered polynomials} of degree $d \ge 2$ which is naturally identified via coefficients 
with $\field^{d-1}$, since the elements of $\Poly_d$ are of the form
$$f(z) = z^d + a_{d-2} z^{d-2} + \cdots + a_{0} \in \field[z].$$
Observe that every  degree $d$ polynomial is affinely conjugate
to a unique element of $\Poly_d$,  modulo conjugacy
given by multiplication by a $(d-1)$-th root of unity.

For $f$ as above,
$$f'(z) = d \cdot \prod_{i=1}^{d-1} (z -c_i),$$
where $\sum_{i=1}^{d-1} c_i =0$
and we say that $(f,c_1, \dots, c_{d-1})$ is  \emph{a critically marked polynomial}. In fact, $\crit(f) := \{ c_1, \dots, c_{d-1} \}$ is the set of critical points of $f$. 
Two critically marked polynomials $(f,c_1,\dots,c_{d-1})$ and
$(g, c_1', \dots, c_{d-1}')$ are \emph{affinely conjugate} if there exists
an affine map $A$ such that $ A \circ f = g \circ A$ and $A(c_i) = c_i'$ for all $i$. When the marking is clear from context, sometimes we simply say that $f$ is a critically marked polynomial.

Recall that the \emph{ramification locus} of $f\in \Polyd$ is 
$$\cR(f):=\{x\in\aoneberk: \deg_xf\ge 2\}.$$
According to Faber~\cite{Faber13} this set  is unbounded, connected
and $\crit(f) = \cR(f)\cap\aone$.  If $\cR(f)\cap\H$ is a locally finite open subtree, then we say that $f$ is a \emph{tame polynomial}.  In this case, $\cR(f)\cup\{\infty\}$ is the convex hull of $\crit(f)\cup\{\infty\}$ in $\poneberk$.
Let $p$ be the residue characteristic of $\C_v$.
If $p=0$, then all polynomials are tame. If $p\neq 0$, then $f$ is tame if and only if for all $x \in \aoneberk$ we have $p \nmid \deg_x f$ or, equivalently, the degree of $f: D \mapsto f(D)$ is not divisible by $p$ for all disks $D \subset \aone$. A thorough study of  $\cR(f)$ is contained in~\cite{Faber13, Faber13II}.

For tame polynomials, we have the following Riemann-Hurwitz formula:
\begin{lemma}[{\cite[Propositions 2.6 and 2.8]{Trucco14}}]\label{lem:RH}
Let $f\in \Polyd$ be a tame polynomial. For any $x\in\aoneberk$, 
$$\deg_xf-1=\sum_{z\in\ol{D}_x\cap\crit(f)}(\deg_z f-1)$$
\end{lemma}

From \cite{Faber13II} we deduce below that
perturbations preserving multiplicities of critical points preserve tameness.
Thus, sometimes it will be convenient to work, in the corresponding parameter space. Namely, consider $k \ge 1$ and  $\mathbf{d}:=(d_1, \dots, d_k)$
where $d_i \ge 2$ for all $i$ and
$$d_1+\cdots+d_k -k = d-1.$$
We denote by $\Poly(\mathbf{d})$ the space
formed by all the pairs $(f,c_1,\dots,c_k)$
where $f \in \Polyd$, $c_i \in \aone$ and $\deg_{c_i} f = d_i$ for all $i=1, \dots,k$. The space $\Poly(\mathbf{d})$ is naturally identified via $(f,c_1,\dots,c_k) \mapsto (c_1,\dots, c_k, f(0))$ with the space formed by the
elements $(c_1,\dots,c_k,b)$ of $ \field^k \times \field$
such that 
$\sum d_i c_i =0$ and $c_i \neq c_j$ if $i \neq j$. Sometimes we abuse of notation and simply say that
$f \in \Poly(\mathbf{d})$ and denote by $c_i(f)$ the corresponding marked critical points.

\begin{lemma}\label{lem:tame}
Tame polynomials are an open {(possibly empty)} subset of $\Poly(\mathbf{d})$ for any $\mathbf{d}$.
\end{lemma}
\begin{proof}
{Let us assume that there exists a tame polynomial in $\Poly(\mathbf{d})$.}
First observe that if $k=1$, then the elements of $\Poly(\mathbf{d})$
    are unicritical polynomials which are tame if and only if $p \nmid d$.
    So we may assume that $k\ge 2$ and let $f \in \Poly(\mathbf{d})$
    be a tame polynomial.  Set $d_0 =d$ and $r_i =\min\{|\binom{d_i}{\ell}|:0\le \ell\le d_i\}$ for $i=0, \dots, k$. Note that $d_i\ge 2$ and $0<r_i\le 1$.
    Let $z_0$ be a type II point of sufficiently large diameter so that
    all the critical points of $f$ are contained in ${D}_{z_0}$. From the Riemman-Hurwitz formula, it follows that $\deg_{z_0} f =d$.
    Consider an open $\dist$-disk $B\subset\H$ of finite radius $R$ centered at $z_0$ such that $B$ contains all the branch points of $\cR(f)$ and for all $1\le i\le k$, there exist 
    $x_i, z_i \in ]c_i(f),z_0] \cap B$ such that $c_i(f)$ is the unique critical point of $f$ in $\ol{D}_{z_i}$ and $x_i \prec z_i$ with $\diam(x_i) < r_i \diam(z_i)$. We may also assume that $R > 1/r_0$ and pick $x _0 \in B$ such that
        $x_0 \succ z_0$ and $\diam(x_0) > \diam(z_0)/r_0$.
     Consider a neighborhood $\cU\subset\Poly(\mathbf{d})$ of $f$ such that for all $g \in \cU$ we have    $g(x)=f(x)$ for all $x \in B$ and $|c_i(g) - c_i(f)| < \diam(x_i)$ for all $i\ge 1$. It follows that $\cR(g) \cap B = \cR(f) \cap B$. Moreover, $c_i(g)\prec x_i$ is the unique critical point of $g$ in $\ol{D}_{z_i}$ for all $i\ge 1$. We now compute the local degrees of $g\in\cU$ at points in $\ol{D}_{x_i}$ applying \cite[Section 4]{Faber13II}. 
Observing that  $d_i=\deg_{z_i}f=\deg_{z_i}g$ and $\deg_{c_i(g)}g=d_i$, we conclude that $\deg_yg=d_i$ for all $y\in[c_i(g),z_i]$. 
Changing coordinates on the source and target by affine maps, we may assume that $c_i(g)=0=g(c_i(g))$ and $z_i=x_G$, which implies that $x_i=x_{0,s_i}$ with $s_i< r_i$. Then $0$ is the unique critical point of $g$ in $\ol{D}_{x_G}$ and $\deg_yg=d_i$ for all $y\in[0,x_G]$. Since $p\nmid d_i$, the conclusion in \cite[Section 4.1]{Faber13II} implies that $\cR(g)\cap\ol{D}_{x_{0,s_i}}=[0,x_{0,s_i}]$.
Let $\ol{D} = \aoneberk \setminus D_{x_0}$. Via the change of coordinates of the form $z \mapsto 1/(az)$, we may also conclude that $\cR(g) \cap \ol{D} = [x_0,\infty[$. 
Now let $ V=  B \cup \ol{D} \cup\left(\bigcup_{1\le i\le d-1} \ol{D}_{x_i}\right)$. Then $\cR(g) \subset V$, since $\cR(g)$ is connected.
    Hence, $\cR(g)$ is the union of the arcs $[c_i(g),\infty[$ for $1 \le i \le k$. Therefore, $g$ is tame for all $g \in \cU$.
\end{proof}

{\begin{remark}
The space $\Poly(\mathbf{d})$ may contain no tame polynomials if $\field$ has positive residue characteristic $p>0$. For example, if an entry $d_i$ of $\mathbf{d}$ is divisible by $p$, then all polynomials in  $\Poly(\mathbf{d})$ are not tame.
\end{remark}}

% The space $\Poly^{cm}_d$ is the set formed by all the  critically marked monic and centered polynomial of degree $d$ which is
% naturally identified with the hyperplane
% $$\{ (c_1,\dots,c_{d-1},b) :  c_1+\cdots+c_{d-1}=0\} \subset \field^d$$
% via $(f, c_1, \dots, c_{d-1}) \mapsto (c_1,\dots,c_{d-1},f(0))$.
% For simplicity,  we refer to elements of
% $\Poly^{cm}_d$ as polynomials $f$ with marked critical points $c_1(f), \dots, c_{d-1}(f)$; moreover, if the label of an element is irrelevant, we sometimes write $c\in\crit(f)$ for a critical point of $f$.

% \subsubsection{Basin of infinity, Julia set and B\"ottcher coordinate}

\smallskip
Given $f \in \Polyd$, the \emph{basin of infinity} of $f$
$$\cB(f) := \{ x \in \aoneberk : f^{\circ n} (x) \to \infty \}$$
is an  open, connected and
unbounded subset of $\aoneberk$.
Its complement
$$\cK(f) := \aoneberk \setminus \cB(f)$$
is the \emph{filled Julia set} of $f$ whose boundary
$$\cJ(f) := \partial \cK(f),$$
is the \emph{Julia set} of $f$.
The Fatou set is $\cF(f) := \aoneberk \setminus \cJ(f)$.
Every connected component of $\cK(f)$ is either a point or a closed disk. In the latter case, every maximal open disk contained in the component of $\cK(f)$ is a bounded Fatou component.

Let $$\cR_\infty (f) := (\cR(f)\cap\cB(f)) \setminus \crit(f)$$
be the set formed by the
non-classical ramification points in the basin of inifinity.
We say that the forward invariant set
$$\cA_f:=\bigcup_{j\ge 0}f^{\circ j} (\cR_\infty(f))$$
is
the \emph{dynamical core of $f$ in $\cB(f)$}.
If in addition we assume that $f$ is tame, then after denoting the postcritical set by $$\post(f):=\bigcup_{j\ge 0}f^{\circ j}(\crit(f))$$ and considering
the open and connected subset of $\aoneberk$
$$\cB_0(f):=\cB(f)\setminus\post(f),$$ we have that
$$\cA_f \subset \sk\cB_0(f).$$ 
Thus, the dynamical core $\cA_f$ of a tame polynomial $f$  is  a locally finite open subtree of $\H$ and any point in $\partial\cA_f$ is either a  postcritical point in $\cB(f)$ or a point in $\cJ(f)\cap\H$.

% subtree of $\aoneberk$.

% whose closure $\ol{\cA_f}$ in $\aoneberk$ is a subtree with $\partial\ol{\cA_f}$ containing all the postcritical points in $\cB(f)$. Hence  $\cB(f)\setminus\cA_f$ is backward invariant and each point in $\cB(f)\setminus\cA_f$ has local degree $1$.

\smallskip
Given a monic and centered polynomial $f(z) = a_0 + a_1 z + \cdots + z^d$
let $$R_f := \max \{ 1, |a_i|^{1/(d-i)}: i=0,\dots, d-2 \}.$$
Then $$x_f := x_{0, R_f}$$
is called the \emph{base point of $f$}.
The base point is characterized by the fact that $\ol{D}_{x_f}$ is the minimal closed Berkovich disk containing $\cK(f)$, see \cite[Section 6.1]{Rivera00}.
It follows that $\deg_{x_f}f=d$ and 
$x_f\in\cR(f)$.
If moreover  $f$ is tame, then every  $c\in\crit(f)$ is contained in $\ol{D}_{x_f}$. Keep in mind that $R_f = |x_f|$.

A polynomial $f\in\Poly_d$ is \emph{simple} if its Julia set is a singleton; {otherwise, we say $f$ is \emph{nonsimple}. For simple $f\in\Poly_d$,} the unique point in  $\cJ(f)$ is the Gauss point
which is also the base point of $f$ and its dynamical core is the open segment connecting the base point to $\infty$. Moreover, 
a tame polynomial is simple if and only if its basin of infinity is critical point free.  {If $f\in\Poly_d$ is nonsimple, } the base point $x_f$  lies in $\cB(f)$ and hence in $\cA_f$, see \cite[Proposition 4.3]{Trucco14}.

\medskip

% \subsubsection{B\"ottcher coordinate}
In an appropriate  coordinate near $\infty$, called the \emph{B\"ottcher coordinate},
a  degree $d$ monic and tame polynomial $f$ acts as the monomial
$z \mapsto z^d$.
More precisely, 
given a tame $f \in \Polyd$
there exists an  analytic isomorphism, 
$$\phi_f : \aoneberk \setminus \ol{D} (0,R_f) \to \aoneberk \setminus
\ol{D} (0,R_f)$$
such that $\phi_f\circ f = \phi_f^d$ and
$\phi_f(z)/z \to 1$ as $\aone \ni z \to \infty$
(e.g.~see~\cite{Ingram13,DeMarco19,Salerno20}).
The germ of $\phi_f$ at infinity is uniquely determined by the above
properties. It follows that $|\phi_f (z)| = |z|$ for all
$z \in \aoneberk \setminus \ol{D} (0,R_f)$.
Although $\phi_f$ does not extend to $\cB(f)$, its absolute value
$|\phi_f|$ extends via the functional equation $|\phi_f|(z)= |\phi (f(z))|^{1/d}$
to a well defined function on the whole basin of infinity $\cB(f)$.
Provided that $f \in \Poly_d$ is tame and nonsimple, it is not difficult to show that
\begin{equation}
  \label{eq:diameter}
    R_f = \max \{ |\phi_f| (c) : c \in \crit(f) \cap \cB(f) \} {>1,}
\end{equation}
and 
{$$R_f^d = \max \{ |\phi_f| (f(c)) : c \in \crit(f) \cap \cB(f) \},$$
see \cite[Proposition 3.4]{Trucco14}.}

To end this subsection, we show that B\"ottcher coordinates $\phi_f$ depend analytically on $f$. We mention here that such analyticity
  appears in \cite[Proposition 2.13]{Favre22} described in terms of Green function. To be more precise, let us introduce 
  the Tate algebra associated to a polydisk (e.g. see~\cite{Fresnel04}).
  For $n \ge 1$, considering 
  $\br=(r_1,\dots,r_n)$ where $r_i \in |\field^\times|$ for all $i$, we denote by $P(\br)$ the rational polydisk in $\field^n$ formed by all
  $(\lambda_1, \dots, \lambda_n) \in \field^n$ such that $|\lambda_i| \le r_i$ for all $i$; then the \emph{Tate algebra} $T_n[\br]$ is
  the one formed by all (formal) power series in $\field \llbracket \lambda_1, \dots, \lambda_n \rrbracket$ convergent in $P(\br)$. Endowed with the sup norm, the Tate algebra  $T_n[\br]$ is a Banach algebra over $\field$ \cite[Theorem 3.2.1]{Fresnel04}. Recall that $p\ge 0$ denotes the residue characteristic of $\mathbb{C}_v$.
\begin{lemma}\label{lem:B-analytic}
 Assume that $p\nmid d$ and pick $\br=(r_1,\dots,r_{d-1})$ with $r_i \in |\field^\times|$ for all $1\le i\le d-1$. Suppose that there exists $R\in |\field^\times|$ with $R>1$ such that $|x_f|<R$ for any $f\in P(\br)\subset\Poly_d\cong\field^{d-1}$. Then the map 
 $$\Phi:P(\br)\times\{|z|\ge R\}\to\{|z|\ge R\},$$
 sending $(f,z)$ to $\phi_f(z)$ is analytic.
 More precisely, setting $\Phi_\infty(f,z)=\Phi(f,z)/z$  we have $\Phi_\infty(f,1/z)\in T_d[\tilde{\br}]$ where $\tilde{\br}=(r_1,\dots,r_{d-1},R^{-1})$.
  \end{lemma}
 \begin{proof}
   First it follows from \cite[Lemma 8]{Ingram13} the map $\Phi$ is well-defined. Now to prove the analyticity, we apply a standard convergence argument as in \cite[Lemma 7]{Ingram13}. For given $f\in P(\br)$ and $|z| \ge R$, define $\Phi_n(f,z)$ to be the $d^n$-th rood of $f^{\circ n}(z)/z^{d^n}$ such that $\Phi_n(f,z)-1\in\mathbb{C}_v\llbracket z^{-1}\rrbracket $. According to~\cite{Ingram13}$\Phi_n(f,z) \to \phi_f(z)/z$ as $n \to \infty$. The lemma will follow after showing that the convergence is uniform.

   Since $p\nmid d$,
  $$\left|\Phi_{n+1}(f,z)-\Phi_n(f,z)\right|=\left|\left(\Phi_{n+1}(f,z)\right)^{d^{n+1}}-\left(\Phi_n(f,z)\right)^{d^{n+1}}\right|.$$
  It follows that 
   \begin{align*}
   \left|\Phi_{n+1}(f,z)-\Phi_n(f,z)\right|&=\left|\frac{f^{\circ{n+1}}(z)}{z^{d^{n+1}}}-\left(\frac{f^{\circ n}(z)}{z^{d^n}}\right)^d\right|=\left|z^{-d^n}\right|\left|f(f^{\circ n}(z))-f^{\circ n}(z)\right|\\
   &=\left|z^{-d^{n+1}}\right|\left|\sum_{j=0}^{d-2}a_j(f)\left(f^{\circ n}(z)\right)^j\right|,
   \end{align*}
   where $a_j(f)$ is the coefficients of $z^j$ in $f$. 
   Let $\| f \|:=\max_{0\le j\le d-2}\{|a_j(f)|\}$ and   $M:=\sup\{||f||: f\in P(\br)\}<\infty$, we conclude that for sufficiently large $n$, 
    \begin{equation*}
    \left|\Phi_{n+1}(f,z)-\Phi_n(f,z)\right| \le\left|z\right|^{-d^{n+1}} \cdot \| f\| \cdot |z|^{d^n(d-2)} = |z|^{-2d^n} \cdot \| f\| \le MR^{-2d^n}. % {\tiny \max_{0\le j\le d-2}\{|a_j(f)|\}} \\ &\le R^{-2d^n}\max_{0\le j\le d-2}\{|a_j(f)|\}=R^{-2d^n}||f||\le MR^{-2d^n}.
    \end{equation*}
  Note that $z \Phi_n (f,z) \to \phi_f(z)$ for all $f \in P(\br)$ and $|z| \ge R$.
    Moreover,  $\{\Phi_{n}(f,1/z) \}$ is a Cauchy sequence in the Banach algebra $T_d[\tilde{\br}]$
    and therefore its limit $\Phi_\infty(f,1/z)$ is analytic.

 % {\tiny \hmn{ Consider  $\br'=(r_1/r_1,\dots,r_{d-1}/r_{d-1})=(1,\dots,1)$. Define $\Psi_1:P(\br')\to P(\br)$ to be the scaling map by the factor $r_{i}$ in the $i$-th coordinate, and define $\Psi_2:\{|z|\ge 1\}\to\{|z|\ge R\}$ to be the scaling map by the factor $R$. Then for any $w\in P(\br')$, identifying $\Psi_1(w)$ to the element in $\Poly_d$, we have that on $P(\br')\times\{|z|\ge 1\}$,
  %   \begin{align*}
 %    \left|\Phi_{n+1}(\Psi_1(w),\Psi_2(z))-\Phi_n(\Psi_1(w),\Psi_2(z))\right|\le MR^{-2d^n}.
 %    \end{align*}
 %    Note that on $P(\br')\times\{|z|\le 1\}$,
 %    \begin{multline*}
 %    \left|\left|\Phi_{n+1}(\Psi_1(w),\Psi_2(1/z))-\Phi_n(\Psi_1(w),\Psi_2(1/z))\right|\right|\\
 %    =\left|\Phi_{n+1}(\Psi_1(w),\Psi_2(1/z))-\Phi_n(\Psi_1(w),\Psi_2(1/z))\right|\le MR^{-2d^n}.
 %    \end{multline*}
 %  Setting that $\tilde{\br}'=(1,\dots,1)$ with $P(\tilde{\br}')\subset\field^{d}$ and noting that 
 % $\Phi_n(\Psi_1(w),\Psi_2(1/z))\in T_d[\tilde{\br}']$, we conclude that $\{\Phi_n(\Psi_1(w),\Psi_2(1/z))\}_{n\ge 1}$ is a Cauchy sequence in $T_d[\tilde{\br}']$, and hence its limit is an element $T_d[\tilde{\br}']$ due to the completeness of $T_d[\tilde{\br}']$. Rewrite this limit in the form $\Phi_\infty(\Psi_1(w),\Psi_2(1/z))$ for some $\Phi_\infty(f,1/z)\in T_d[\tilde{\br}]$. Observing that $\Phi_\infty$ is the limit of $\Phi_n$ as $n\to\infty$, we obtain the conclusion.}}
%and hence $\{\Phi_n(f,1/z)\}_{n\ge 1}$ is a uniform Cauchy sequence on $P(\tilde{\br})$. Since $\Phi_n(f,1/z)\in T_d[\tilde{\br}]$, the conclusion follows immediately from the completeness of $T_d[\tilde{\br}]$.}
  \end{proof}

%\jk{base point, the tree $\cA_f$}

%\hmn{define analytic family, same B\"ottcher, passive}

\subsection{Analytic family}
\label{sec:analytic-family}
We consider analytic families parametrized by an open  disk $\Lambda$ in $\field$ of radius in $|\field^\times|$.
% \footnote{For the purpose of this paper we could work with $\Lambda$ being a boundaryless analytic space. (e.g. see~\cite[page 20]{Temkin})}
A family $\{f_\lambda\}_{\lambda\in\Lambda}\subset\Poly_d$ is \emph{analytic} if
$$f_\lambda(z) = z^d + \sum_{i=0}^{d-2} a_i (\lambda) z^i$$
where $a_i(\lambda)$ is analytic for all $i$ (i.e. a power series in $\lambda$
which converges for all $\lambda \in \Lambda$).
In this case, we say that $\{(f_\lambda,c_1(\lambda), \dots, c_{d-1}(\lambda))\}$ is a \emph{critically marked
  analytic family} if $(f_\lambda,c_1(\lambda), \dots, c_{d-1}(\lambda))$
is a critically marked polynomial for all
$\lambda \in \Lambda$ and $c_i(\lambda)$ is analytic in $\Lambda$ for all $i$.
To simplify notation, we sometimes just label by $\lambda$
the objects associated to $f_\lambda$.
Also we will often omit the markings $c_i(\lambda)$ from the notation and
simply say that $\{f_\lambda\}$ is a critically marked analytic family implicitly assuming that the markings are $c_1(\lambda), \dots, c_{d-1}(\lambda)$.

When  $\{f_\lambda\}_{\lambda\in\Lambda}\subset\Poly_d$ is an analytic family and
  $c: \Lambda \to \field$ is an analytic function such that
$c(\lambda)$ is a critical point of $f_\lambda$ for all $\lambda$ we
say that $c(\lambda)$ is a \emph{marked critical point} of the family $\{f_\lambda\}$.

\begin{definition}
  Let $\{f_\lambda\}_{\lambda\in\Lambda}$ be an analytic  family with a
  marked critical point  $c(\lambda)$.
  We say that $c(\lambda)$ is a \emph{passive critical point in $\Lambda$} if $\{\lambda\in\Lambda : c(\lambda)\in\cB(\lambda)\}=\Lambda$ or $\emptyset$.
We say that  a critically marked analytic family parametrized by $\Lambda$ is \emph{passive in $\Lambda$}
  if all its critical points are passive in $\Lambda$. 
\end{definition}

We emphasize that our definition of
passive is not a local property, it depends on the domain $\Lambda$ of
the analytic family $\{f_\lambda\}$. 

For an analytic family $\{f_\lambda\}_{\lambda\in\Lambda}$, denote by  $\phi_{\lambda}: \aoneberk \setminus \ol{D} (0,R_\lambda) \to \aoneberk \setminus\ol{D} (0,R_\lambda)$ the B\"ottcher coordinate of $f_\lambda$. Given
$\lambda_0, \lambda_1 \in \Lambda$ and  $R=\max\{R_{\lambda_0},R_{\lambda_1}\}$
we say that 
$\phi_{\lambda_1}^{-1}\circ\phi_{\lambda_0}: \aoneberk \setminus \ol{D} (0,R) \to \aoneberk \setminus\ol{D} (0,R)$
is the \emph{B\"ottcher coordinate change} between the dynamical space
of $f_{\lambda_0}$ and $f_{\lambda_1}$. 

\begin{definition}\label{def:constant-B}
  Consider a passive marked critical point
  $c(\lambda)$ in $\cB(\lambda)$
  of an analytic family $\{f_\lambda\}$ of tame polynomials
  parametrized by $\Lambda$.
  We say that the \emph{B\"ottcher coordinate of $c(\lambda)$ is constant} if
  for all $\lambda_0, \lambda_1 \in \Lambda$, there exists $n$ such that
  the corresponding B\"ottcher coordinate change
  sends $f^{\circ n}_{\lambda_0}(c(\lambda_0))$ to $f^{\circ n}_{\lambda_1}(c(\lambda_1))$. 
\end{definition}

As an immediate consequence of~\eqref{eq:diameter} we have that
base points of passive families are constant:

\begin{lemma}
  \label{l:passive-constant}
  Let $\{f_\lambda\}_{\lambda\in\Lambda}$ be a passive critically marked analytic family of tame polynomials parametrized by a disk $\Lambda$.
  Then the base point of $f_\lambda$ is independent of
  $\lambda\in\Lambda$. 
\end{lemma}

\begin{proof}
  We may assume that the maps involved are nonsimple.
  Pick an element $\lambda_0 \in \Lambda$ and apply \eqref{eq:diameter}
  to find a critical point $c_i(\lambda_0) \in \cB(\lambda_0)$ such
  that $R_{\lambda_0} = |\phi_{\lambda_0}|(c_i(\lambda_0))$.
  Keep in mind that $|c_i(\lambda)| \le R_\lambda$ for all $\lambda$.
  Let $v(\lambda) =f_\lambda (c_i(\lambda))$. Then
  $|v(\lambda_0)| = R^d_{\lambda_0}$.
  Observe that
  $v(\lambda)-c_i(\lambda) \neq 0$ for all $\lambda$, since the critical points are all passive. Therefore, $|v(\lambda)-c_i(\lambda)|$ has constant
  value  $R_{\lambda_0}^d$.
  Thus, for all $\lambda$,
  $$R_{\lambda_0}^d = |v(\lambda)-c_i(\lambda)| \le R_\lambda^d$$
  where the last inequality follows from
  $|c_i(\lambda)| \le R_\lambda$ and $|v(\lambda)| \le R_\lambda^d$.
  Since this occurs for all $\lambda_0 \in \Lambda$, we have that $R_\lambda$ is constant.
\end{proof}

{Observe that if the base point $x_\Lambda$ of an analytic
family $\{f_\lambda\}$ of tame polynomials is independent of $\lambda$, then the B\"ottcher coordinate change
$\phi^{-1}_\lambda \circ \phi_{\lambda_0}(z)$ is also analytic in $\lambda$ for all
$z$ such that $|z| > |x_\Lambda|$ {(see Lemma \ref{lem:B-analytic})}.
Hence, if a passive marked critical point
$c(\lambda)$ in $\cB(\lambda)$ has constant B\"ottcher coordinate, 
then the B\"ottcher coordinate change
sends $f^{\circ m}_{\lambda_0}(c(\lambda_0))$ to $f^{\circ m}_{\lambda}(c(\lambda))$ for all $m \ge 1$ such that $|f^{\circ m}_{\lambda_0}(c(\lambda_0))| > |x_\Lambda|$.}

% Taking  all critical points in $\cB(\lambda)$ into account, we use the following definition:

% \begin{definition}
%  We say the \emph{B\"ottcher coordinates of escaping critical points of $f_\lambda$ are constant} if $\{f_\lambda\}_{\lambda\in\Lambda}$ is passive and the B\"ottcher coordinate of each critical point in $\cB(\lambda)$ is constant
% \end{definition}

% \begin{proof}
%   We may assume that $f_{\lambda_0}$ is nonsimple for some $\lambda_0$, otherwise, $x_\lambda=x_G$. Following \cite[Section 3]{Trucco14}, we have that there exists a critical point $c(\lambda_0)$ such that $|c(\lambda_0)|=\diam(x_{\lambda_0})$ and hence $c(\lambda_0)\in\cB(\lambda_0)$. It follows that
%   $|c(\lambda_0)|^{d^n} = |f^{\circ n}_{\lambda_0}(c(\lambda_0))|$
%   for all $n \ge 1$.
%   B\"ottcher coordinates
%   are asymptotic to the identity,  therefore $|f^{\circ n}_{\lambda_0}(c({\lambda_0}))| = |\phi_{\lambda_0}(f^{\circ n}_{\lambda_0}(c(\lambda_0)))|$. Since the
%   B\"ottcher coordinate of $c(\lambda)$ is constant,  
%   given any $\lambda_1 \in \Lambda$, we have $|c(\lambda_1)|=|c(\lambda_0)|$ which yields that  $\diam(x_\lambda)$ is constant and the conclusion follows.
% \end{proof}

%\begin{lemma}
%If $\{f_\lambda\}$ has constant B\"ottcher coordinates, then $\{f_\lambda\}$ is passive
%\end{lemma}
%\begin{proof}
%The constant B\"ottcher coordinates assumption implies that if $c(\lambda_0)\in\cB(\lambda_0)$ for some $\lambda_0\in\Lambda$, then $c(\lambda)\in\cB(\lambda)$ for all $\lambda\in\Lambda$.
%\end{proof}

\section{Extendable  and Analytic Conjugacies}
\label{sec:extendable}
The aim of this section is to discuss the notion of extendable conjugacies and prove Theorem~\ref{ithr:D}.

%For the rest of the paper we let $d \ge 2$. 

\subsection{Extendable conjugacies}\label{sec:conjugacy}
We shall investigate conditions under which the conjugacy $\phi^{-1}_g \circ \phi_f$ furnished near $\infty$
by the B\"ottcher coordinates of two polynomials $f,g \in \Polyd$
extends to an analytic conjugacy $\varphi: \cB (f) \to \cB(g)$.

\begin{definition}\label{def:conjugacy}
  Consider tame polynomials $f, g \in \Polyd$ and locally finite unbounded
  open subtrees   $\cT_f, \cT_g \subset \H$  such that $f(\cT_f) \subset \cT_f$ and $g(\cT_g) \subset \cT_g$. A $\dist$-isometry $h: \cT_f \to \cT_g$ %that conjugates $f$ with $g$ (i.e. $h \circ g = f \circ h$) is called an \emph{extendable conjugacy} if
  is an  \emph{extendable conjugacy} if
 \begin{enumerate} 
  \item (conjugacy) $h \circ g = f \circ h$,
  \item (B\"ottcher coordinate change)
  $\phi_g^{-1} \circ \phi_f (x) = h(x)$ for all $x \in \cT_f$ in a neighborhood of $\infty$, and
 \item (locally a translation)
  for all $x \in \cT_f$, there exists a translation
  $\tau_x (z) =z +b_x$ such that $\tau_x (y) = h(y)$ for all $y\in\cT_f$ in a neighborhood of $x$.
  \end{enumerate}
\end{definition}

% \begin{remark}
% It might be interesting to study the class of isometries which are locally M\"obius. However, in our setting for polynomials, the stronger condition that is locally a translation will be particularly suited for our purpose.
% \end{remark}
%\jk{Remark. it might be interesting to study the class of isometries which are locally Moebius. However, for us, the stronger condition that is locally a translation will be particularly suited for our purpose.}

%\jk{After lemmas 3.4 and 3.5 we may prove Lemma: admissible conjugacy iff tangent map formulation? do we need it?}
  
The following result implies that an extendable conjugacy has a well-defined tangent map.
\begin{lemma}
  \label{l:unique-admissible}
  Let $\cT_0,\cT_1$ be two  locally finite open subtrees in $\H$ and let  $\iota: \cT_0 \to \cT_1$ be locally a translation.
  Consider $x \in \cT_0$ and two translations $\tau$ and $\tau'$ which agree with $\iota$ in a neighborhood of $x$. Then $T_x \tau = T_x \tau'$. 
\end{lemma}

\begin{proof}
Noting that the valence of $\cT_0$ at $x$ is at least $2$, since $\tau=\iota=\tau'$ in a neighborhood of $x$, we obtain that $T_x \tau$ and $T_x \tau'$ agree in at least two directions in $T_x\aoneberk$ and therefore $T_x \tau = T_x \tau'$.
  %The existence follows from the definition. The uniqueness is a direct consequence of the following straightforward fact.
  %For any $x \in \aoneberk$, if two translations
  %$\tau_a, \tau_b$ are such that $\tau_a(x) = \tau_b(x)$ and
  %$T_x \tau_a$,  $T_x \tau_b$ agree in at least two directions, then $T_x \tau_a = T_x \tau_b$.
\end{proof}

By Lemma \ref{l:unique-admissible}, if $h$ is an extendable conjugacy as in Definition \ref{def:conjugacy}, given $x\in\cT_f$ and setting $y=h(x)$, we denote by $T_x h : T_x \aoneberk \to T_y \aoneberk$ the unique map that agrees with the tangent map of a translation {$\tau_x$} locally coincident with $h$, that is $T_x h=T_x\tau_x$.

\subsection{From analytic to extendable conjugacies}
\label{sec:analytic2extendable}
In this subsection we discuss general results about analytic maps between open subsets of $\aoneberk$ and  prove one direction of Theorem~\ref{ithr:D} (see Corollary \ref{c:ithr-D-easy}). %Then in Section~\ref{s:maps-annuli} we discuss well known results about tame maps between annuli that  will be needed for the other direction. 

An open annulus in $\aoneberk$ is a set $A$ of the form $D (a,r) \setminus \ol{D}(a,s)$ for some $0 < s < r$. The image of a non-constant analytic map $\psi: A \to \aoneberk$ is either a disk or an  annulus (c.f.\cite[Proposition~9.44 and Lemma~9.45]{Baker10}). In the latter case, $\psi$  has a well defined degree, and in a certain sense,  
$T_x \psi$ is ``constant'' along type II points $x \in \sk A$:

\begin{lemma}
  \label{l:constant-skeleton}
  Let $A_0$ and $A_1$ be two open annuli contained in $\aoneberk$, and 
  suppose that $\psi : A_0 \to A_1$ is a surjective analytic map of degree $\delta$.
  Then there exist $\sigma \in
  \{+\delta,-\delta\}$ and $0\not=a\in \C_v$ such that 
for arbitrary $b\in\aone$ in the bounded component of $\aoneberk\setminus A_1$ and arbitrary $c\in\aone$ in the bounded component of $\aoneberk\setminus A_0$, setting $\gamma(z) = a (z-c)^\sigma + b$, we have that  $\psi(x)=\gamma(x)$ and $T_x \psi = T_x \gamma$
  for all $ x \in \sk (A_0)$.  
\end{lemma}

\begin{proof}
  Pick any $c \in \aone$ in the bounded component of $\aoneberk\setminus A_0$ and any $b \in \aone$ in the bounded component of $\aoneberk\setminus A_1$. Modulo the translation $z\mapsto z-c$ in the domain and the translation $z\mapsto z-b$ in the range, we may assume that $\sk(A_0)$ and $\sk(A_1)$ are contained in $]0,\infty[$.
  Then, from the Mittag-Leffler decomposition of $\psi$ (see~\cite[Proposition~2.2.6]{Fresnel04}), we have $\psi (z) = \sum_{-\infty}^{+\infty} a_n z^n$.
  Moreover, for some $\sigma \in \{+\delta,-\delta\}$ and for all $r$
  such that $|z| = r$ is contained in $A_0$, we have
  that $|\psi(z)| = |a_{\sigma}|r^\sigma$ and $|a_j| r^j < |a_\sigma|r^\sigma$ for all $j \neq  \sigma$ since $\psi$ maps $\sk(A_0)$ onto $\sk(A_1)$ and has degree $\delta$.
  Hence, $$|\psi (z) - a_\sigma z^\sigma| < |a_\sigma|r^\sigma.$$
  Let $\gamma (z) = a_\sigma z^\sigma$. It follows that $\psi(x)=\gamma(x)$ and $T_x \psi = T_x \gamma$ for all $ x \in \sk(A_0)$. 
  \end{proof}

 % \jk{maybe as corollary of the above:}

\begin{corollary}
  \label{l:t-criteriumB}
  Let $A_0$ and $A_1$ be two open annuli in $\aoneberk$, and suppose 
  that $\psi_1, \psi_2: A_0 \to A_1$ are surjective analytic maps.
  If there exists a type II point $x_0 \in \sk A_0$ such that $\psi_1(x_0) = \psi_2(x_0)$ and
  $T_{x_0} \psi_1 = T_{x_0} \psi_2$, then $\psi_1(x) = \psi_2(x)$ and
  $T_{x} \psi_1 = T_{x} \psi_2$ for all $x \in \sk A_0$. 
\end{corollary}

\begin{proof}
Modulo affine maps in the domain and the target, we can assume that $\sk(A_0)$ and $\sk(A_1)$ are contained in $]0,\infty[$ and $x_0=x_G=\psi_1(x_0) = \psi_2(x_0)$. Since the point $0$ is contained in both the bounded component of $\aoneberk\setminus A_0$ and the bounded component of $\aoneberk\setminus A_1$, by Lemma \ref{l:constant-skeleton}, there exists $\gamma_1(z)=az^\delta$ and $\gamma_2(z)=a'z^{\delta'}$ such that $\gamma_1(x)=\psi_1(x)$, $\gamma_2(x)=\psi_2(x)$ and $T_{x}\gamma_1=T_{x}\psi_1$, $T_{x}\gamma_2=T_{x}\psi_2$ for all $x\in\sk(A_0)$. Since $x_0=x_G\in\sk(A_0)$ and $\gamma_1(x_0)=x_0=\gamma_2(x_0)$, we obtain $|a-a'|<|a|=|a'|=1$ and $\delta=\delta'$. Then for any $r$ with $|z|=r$ contained in $A_0$, we have
$$|\gamma_1(z)-\gamma_2(z)|=|(a-a')z^\delta|<|\gamma_1(z)|.$$
It follows that $|(\gamma_1 -\gamma_2)(x)|<|\gamma_1(x)|$ and $T_x\gamma_1 = T_x\gamma_2$ for all $ x \in \sk(A_0)$. Thus the conclusion holds.
\end{proof}

 %  \begin{corollary}
%     Suppose that $\psi : A_0 \to A_1$ is an analytic isomorphism between two open annuli contained in $\aoneberk$.
%     If there exists a type II point $x_0$ in $\sk(A_0)$ such that $T_{x_0} \psi = T_{x_0} \tau_c$ for some $c \in \C_v$,
% then $\psi : \sk (A_0) \to \sk(A_1)$ is an admissible isometry.
%   \end{corollary}

We will also need the following observation:

\begin{lemma}
  \label{l:constant-ball}
  Suppose that $\psi_1, \psi_2: U \to \aoneberk$
   are analytic maps defined in a neighborhood $U$ of a type $II$ point $x_0$.
  If $\psi_1 (x_0) = \psi_2(x_0)$ and $T_{x_0} \psi_1 = T_{x_0} \psi_2$, then there exists $\varepsilon >0$ such that for all $x\in U$ with $\dist (x, x_0) < \varepsilon$ we have that $\psi_1(x)=\psi_2(x)$ and
  $T_x \psi_1 = T_x\psi_2 $. 
\end{lemma}

\begin{proof}
  Let $y_0 = \psi_1(x_0) = \psi_2(x_0)$. Taking into account that $x_0$ is a type II point,
  given such  maps $\psi_1, \psi_2$ we have that
  $|\psi_1(z) - \psi_2(z)|_{x_0} < \operatorname{diam} (y_0)$. Since $\operatorname{diam} (y)$ is continuous with respect to
    the metric topology in $\H$ and $|\psi_1(z) - \psi_2(z)|_{x}$ is continuous in the Berkovich topology which is weaker than the metric topology, it follows
    that $|\psi_1(z) - \psi_2(z)|_{x} < \operatorname{diam} (\psi_1(x))$ for all $x$ in a $\H$-metric disk around $x_0$. Hence, $\psi_2(x) = \psi_1(x)$ and
        $T_x \psi_2 = T_x \psi_1$ for all $x$ in this $\dist$-disk. 
\end{proof}

Analytic isomorphisms between open and connected sets which agree with a translation in a neighborhood of an skeleton point of the domain, are locally a translation at all skeleton points:

\begin{proposition}
  \label{p:iso-adm}
  For $b \in \field$, let $\tau_b(z) = z+b$. 
  Consider two open and connected sets $U_0, U_1\subsetneq \aoneberk$.
  Assume that  $\psi: U_0 \to U_1$ is an analytic isomorphism such that $\psi(x_0)=\tau_b(x_0)$ and 
  $T_{x_0} \psi = T_{x_0} \tau_b$ for some $x_0 \in \sk U_0$ and
  some $b \in \field$.
  Then for any $x\in \sk U_0$, there exist $b_x\in \field$ and a neighborhood $V\subset \sk U_0$ of $x$ such that for all $y\in V$, $\psi(y)=\tau_{b_x}(y)$ and
  $T_{y} \psi = T_{y} \tau_{b_x}$.
\end{proposition}

\begin{proof}
  Let $\Gamma \subset \sk U_0$ be the set formed by  all $x \in \sk U_0$
  with the following property:
  there exist $b_x\in\field$ and a neighborhood $V\subset \sk U_0$ of $x$ such that for
  all $y \in V$ we have $\psi(y)=\tau_{b_x}(y)$ and $T_y \psi = T_y \tau_{b_x}$. By Lemma~\ref{l:constant-ball}, if a branch point $x'\in\sk U_0$ is contained in $\Gamma$ (resp. $\sk U_0\setminus \Gamma$), then there exists an open hyperbolic ball  $V'\subset U_0$ around $x'$ such that $V'\cap \sk U_0$ is contained in $\Gamma$ (resp.
  in $(\sk U_0) \setminus \Gamma$). Given any two branch points $x_1$ and $x_2$ of $\sk U_0$ such that $]x_1,x_2[$ contains no branch point of $\sk U_0$, by Lemma~\ref{l:constant-skeleton} and \ref{l:constant-ball}, the segment $]x_1,x_2[$ is either contained in $\Gamma$ or in $(\sk U_0) \setminus \Gamma$. Therefore, $\Gamma$ is a non-empty clopen  subset of $\sk U_0$ and the proposition follows.
\end{proof}

% Observe that the map $\psi$ in Proposition \ref{p:iso-adm} descends an isometry $\psi: \sk U_0\to\sk U_1$. The following is a consequence of Proposition \ref{p:iso-adm} in the case that $\psi$ is an analytic conjugacy on the basins of infinity for two tame polynomials in $\Poly^{cm}_d$.
Given a tame polynomial $f \in \Poly_d$, 
recall from Section \ref{sec:poly} that the dynamical core
$\cA_f$ is  contained in the skeleton of the open and connected set $\cB_0(f)$.

\begin{corollary}
  \label{c:ithr-D-easy}
  Let $f, g\in\Polyd$ be two tame polynomials. Assume that $\psi: \cB (f) \to \cB(g)$ is an analytic conjugacy between $f: \cB (f) \to \cB(f)$ and $g: \cB (g) \to \cB(g)$ that extends the corresponding B\"ottcher coordinate change.
  Then $\psi: \cA_f \to \cA_g$ is an extendable conjugacy between $f: \cA_f \to \cA_f$ and $g: \cA_g \to \cA_g$.
\end{corollary}

\begin{proof}
  By Definition \ref{def:conjugacy}, it suffices to show $\psi$ is locally a translation on $\cA_f$. Observe that $c$ is a critical point
  of $f$ in $\cB(f)$ if and only if  $\psi(c)$ is a critical point
  of $g$ in $\cB(g)$. Hence, $\psi:\cB_0(f)\to\cB_0(g)$ is an analytic isomorphism. Since $\cA_f \subset \sk\cB_0(f)$, by Proposition \ref {p:iso-adm}, $\psi$ is locally a translation on $\cA_f$.
\end{proof}

Corollary~\ref{c:ithr-D-easy} establishes one direction of Theorem~\ref{ithr:D}.

\subsection{From extendable to analytic conjugacies}
\label{s:maps-annuli}

In this subsection, we discuss results about maps between open sets and prove the remaining direction of Theorem~\ref{ithr:D}. 

Recall that $p \ge 0$ denotes the characteristic of the residue field of $\field$.
We begin with the following well-known facts about prime-to-$p$ \'etale maps between annuli. Proofs are supplied for the sake of completeness.

\begin{lemma}
  \label{l:isomorphic-d}
Let $\delta \ge 1$ be an integer not divisible by $p$, and  let $A_0 = \{ r^{1/\delta} < |z| < s^{1/\delta} \}$ and $A_1 = \{ r < |z| < s \}$ be two open annuli in $\aoneberk$ for some $0 < r < s$. Consider an analytic map  $\psi: A \to A_1$ of degree $\delta$,
  where $A \subset \aoneberk$ is an open annulus.
  Then there exists an analytic isomorphism $\varphi: A \to A_0$
  such that $\varphi (z)^\delta = \psi(z)$. Moreover, $\varphi$ is unique up to multiplication by $\mu \in \field$ where $\mu^\delta =1$.
\end{lemma}

\begin{proof}
  Write $A = \{ r' < |z-a| < s' \}$ for some $a\in\field$. Then there exists $b_n\in\field$ such that
  $$\psi(z) = \sum_{-\infty}^{+\infty} b_n (z-a)^n$$
  where, for all  $r' < t < s'$, either $|b_\delta| t^\delta > |a_n| t^n$ for all $n \neq \delta$  or $|b_{-\delta}| t^{-\delta} > |b_n| t^n$ for all $n \neq -\delta$.
  Without loss of generality, assume the former. Then
  $$\psi(z) = b_\delta (z-a)^\delta (1 + \varepsilon(z))$$ with $|\varepsilon(z)| < 1$ for all $z \in A$. Consider
  $\beta \in \C_v$ such that $\beta^\delta =b_\delta$. Since 
  $p$ does not divide $d$, there exists a
  unique function $\gamma(z)$ such that $\gamma(z)^\delta = 1 + \varepsilon(z)$
  and $|\gamma(z) -1| <1$. Then $\varphi(z) = \beta (z-a) \gamma(z)$ is such that $\varphi(z)^\delta = g(z)$. The lemma follows.  
\end{proof}

\begin{corollary}
  \label{c:identity}
  Suppose that $\psi: A  \to A_1$ is an analytic map of degree $\delta \ge 2$ between two open annuli $A, A_1\subset\aoneberk$ such that 
  $\delta$ is not divisible by $p$. 
  Assume that $\psi_1: A \to A $ is an analytic isomorphism such that the following hold:
  \begin{enumerate}
  \item $$\psi \circ \psi_1 = \psi.$$
  \item $\psi_1(x) = x$ and $T_x \psi_1 = \operatorname{id}$ for some type II point $x$ in $\operatorname{sk}(A)$.
  \end{enumerate}
  Then $\psi_1 = \operatorname{id}$.
\end{corollary}

\begin{proof}
  After translation we may assume $A_1 =
  \{ r < |z| < s \}$. Let
  $A_0$ and $\varphi: A \to A_0 $ be as in Lemma \ref{l:isomorphic-d}. Then $\varphi(z)^\delta=\psi(z)$. It follows that $(\varphi (\psi_1(z)))^\delta =\psi(\psi_1(z))$. By assumption (1), we have $(\varphi (\psi_1(z)))^\delta=\psi(\psi_1(z))= \psi(z)=\varphi(z)^\delta$. Thus $\varphi \circ \psi_1(z) = (\mu \varphi) (z)$ for some
  $\mu\in\field$ with $\mu^\delta =1$.
  %f $T_x \psi_1 = \operatorname{id}$ for some type II point $x$ in $\operatorname{sk}(A_0)$, then
  By assumption (2), we conclude $\varphi(x)=(\mu\varphi)(x)$ and 
  $T_x \varphi = T_x (\mu\varphi)$. Thus $\mu=1$. It follows that  $\varphi \circ \psi_1 = \varphi$ which implies  $\psi_1 = \operatorname{id}$.
\end{proof}

\begin{corollary}
\label{c:extension-annulus}
Let $A,A'$ and $A_1$ be open annuli in $\aoneberk$. Suppose that $\psi_1: A \to A_1$ and $\psi_2: A' \to A_1$ are
degree $\delta \ge 1$ analytic maps such that 
  $\delta$ is not divisible by $p$. 
Then there exist exactly $\delta$ isomorphisms $\varphi: A \to A'$ such that
$\psi_2\circ \varphi =\psi_1$. Moreover, if $B_1$ is a subannulus of $A_1$ with
$\sk(B_1) \subset \sk(A_1)$, then every isomorphism $\psi: \psi_1^{-1} (B_1) \to
\psi_2^{-1}(B_1)$ such that $\psi_2 \circ \psi =\psi_1$ extends to a unique 
isomorphism $\varphi: A \to A'$ such that $\psi_2\circ \varphi =\psi_1$. 
\end{corollary}

\begin{proof}
  After change of coordinates, unique modulo multiplication by a $\delta$-th root of unity, $\psi_1$ and $\psi_2$ become
  $z \mapsto z^\delta$. It follows that there are exactly $\delta$ isomorphisms
  $\varphi$ as in the statement. Similarly, there are exactly $\delta$
  isomorphisms $\psi$ as in the statement. Thus,
every such $\psi$ is the restriction of an isomorphism $\varphi$.
  \end{proof}
% After translation we may assume $A_1 =
%   \{ r < |z| < s \}$ and  $A_0 =
%   \{ r^{1/\delta} < |z| < s^{1/\delta} \}$. For each $1\le j\le 2$, by Lemma \ref{l:isomorphic-d}, the isomorphisms $\varphi_j: A \to A_0 $ such that $(\varphi_j(z))^\delta=\psi_j(z)$ is unique up to multiplication by a $\delta$-th root of unity. For any isomorphism $\varphi: A \to A'$ satisfying
% $\psi_2\circ \varphi =\psi_1$, we have $(\mu_2\varphi_2(\varphi(z)))^\delta=(\mu_1\varphi_1(z))^\delta$. It follows that $\mu_2\varphi_2(\varphi(z))=\mu\mu_1\varphi_1(z)$ for some $\mu\in\field$ with $\mu^\delta=1$, and hence 
% $$\varphi=\left(\frac{\mu_2}{\mu\mu_1}\varphi_2\right)^{-1}\circ\varphi_1.$$
% Since $\mu_2/(\mu\mu_1)$ varies in the $\delta$-th roots of unity, we have exactly $\delta$ such $\varphi$.

% It follows that there are exactly $\delta$ isomorphisms
% $\psi$ as in the statement. Thus every such $\psi$
% is the restriction of an isomorphism $\varphi$.
% %There are exactly $\delta$ isomorphisms $\overline{\psi}: A_0 \to A_0'$   such that $g \circ \overline{\psi} = f$ and  $\delta$ isomorphisms ${\psi}: B_0 \to B_0'$ such that $g \circ {\psi} = f$. It follows that every one of the latter is a restriction of a former.

Now we provide a criterium for the tangent maps of two analytic functions
to agree at a type II point.

\begin{lemma}
  \label{l:t-criteriumA}
  Let $D\subset\aoneberk$ be an open Berkovich disk and $x_0\in D$ be a type II point.
  Suppose that $\psi_1, \psi_2: D \to \aoneberk$ are analytic maps satisfying the following:
    \begin{enumerate}
  \item  $p\nmid\deg_{x_0}\psi_1$ and $p\nmid\deg_{x_0}\psi_2$,
   \item $\deg_{\vec{v}} \psi_1 = \deg_{\vec{v}} \psi_2$ for all $\vec{v} \in T_{x_0} \aoneberk$,
   %\item $T_{x_0} \psi_1 (\vec{v}_0) = T_{x_0} \psi_2 (\vec{v}_0)$ for at least one direction $\vec{v}_0 \in  T_{x_0} \aoneberk\setminus\{\vec{v}_{x_0}(\infty)\}$,   
  \item there exists $x\in\aoneberk$ with $x_0\precneqq x$ such that 
  \begin{enumerate}
  \item $\psi_1 (y_0) = \psi_2(y_0)$ for all $x_0\prec y_0 \prec x$, and
  \item $T_{y_0} \psi_1 = T_{y_0} \psi_2$ for all $x_0\precneqq y_0 \prec x$,
  \end{enumerate}
  \item $T_{x_0} \psi_1 (\vec{v}_0) = T_{x_0} \psi_2 (\vec{v}_0)$ for at least one direction $\vec{v}_0 \in  T_{x_0} \aoneberk\setminus\{\vec{v}_{x_0}(\infty)\}$.   
 %\item $\deg_{\vec{v}}T_{x_0}\psi_1 = \deg_{\vec{v}}T_{x_0}\psi_2$ for all $\vec{v} \in T_{x_0} \aoneberk$. 
  \end{enumerate}
  Then $T_{x_0} \psi_1 = T_{x_0} \psi_2$. 
\end{lemma}
\begin{proof}
  From (3a), continuity yields that $\psi_1(x_0) = \psi_2(x_0)$. Using the same
  coordinate changes for $\psi_1$ and $\psi_2$, 
  we may assume that $x_0$ and $\psi_1(x_0)=\psi_2(x_0)$ are the Gauss point. Thus, it is sufficient to show that the reductions
  $\tilde{\psi}_1$ and $\tilde{\psi}_2$ of $\psi_1$ and $\psi_2$ coincide. 
  While statement (1) implies that  the polynomials
  $\tilde{\psi}_1$ and $\tilde{\psi}_2$
   have finitely many critical points, statement (2) implies that
   $\tilde{\psi}_1$ and $\tilde{\psi}_2$  have the same critical points counting multiplicities. In particular,  $\deg_{x_0}\psi_1=\deg_{x_0}\psi_2$.
   From (3b), it is not difficult to conclude that the leading coefficients of $\tilde{\psi}_1$ and $\tilde{\psi}_2$ coincide.
   From (4), the constant terms of  $\tilde{\psi}_1$ and $\tilde{\psi}_2$ agrees. Therefore,  $\tilde{\psi}_1 =\tilde{\psi}_2$.
  %  $T_{x_0} \psi_1$ and $T_{x_0} \psi_2$ agree. 
  % Statement (3) implies that $T_{x_0} \psi_1$ and $T_{x_0} \psi_2$ have same leading coefficients. Indeed, for any $y_0\in]x_0,\infty[$ sufficiently close to $x_0$, we have that $\deg_{y_0}\psi_1=\delta=\deg_{y_0}\psi_2$ and hence for $j=1, 2$, the leading coefficients of $T_{x_0}\psi_j$ and  $T_{y_0}\psi_j$ are same up to multiplying a fixed constant that is independent of $j$, thus the leading coefficients of $T_{x_0}\psi_1$ and  $T_{x_0}\psi_2$ coincide by Statement (3). Statement (4) implies that $T_{x_0} \psi_1$ and $T_{x_0} \psi_2$ have same constant terms. Therefore, the conclusion follows.
\end{proof}

%\subsection{Analytic Continuation}

Through an analytic continuation argument along $\cA_f$,
we will prove the following:

\begin{proposition}\label{p:ithr-D-hard}
Let $f, g \in \Poly_d$ be tame polynomials. If $h: \cA_f \to \cA_g$  is an extendable conjugacy  between  $f: \cA_f \to \cA_f$ and $g: \cA_g \to \cA_g$, then there exists an  analytic conjugacy $\phi: \cB(f)\to\cB(g)$ between $f: \cB (f) \to \cB(f)$ and  $g: \cB (g) \to \cB(g)$
  which %is asymptotic to the identity. 
  agrees with  the B\"ottcher coordinate change near infinity
  and with $h$ on $\cA_f$.
\end{proposition}

Given a tame polynomial $f \in \Poly_d$, recall from Section~\ref{sec:poly} that $\phi_f$ denotes the B\"ottcher coordinate and that
$|\phi_f|$ extends to $\cB(f)$. 
For $s>1$, let $$V_f (s) := \{ z \in \aoneberk : |\phi_f| (z) > s \}$$
and $$\cA_f (s) := \cA_f \cap V_f (s).$$
Observe that $\cA_f (s)$ is forward invariant. Moreover, $\cA_f(s)$ is contained in the skeleton of $$V'_f(s):=V_f(s) \setminus \post(f).$$

\begin{proof}[Proof of Proposition \ref{p:ithr-D-hard}]
 % Suppose that $f, g \in \polycmd$ and $h: \cA_f \to \cA_g$ is an admissible conjugacy 

For all $z \in \aoneberk$ such that $|z|$ is sufficiently large, we have  $|\phi_f (z)| = |z| = |\phi_g (z)|$.
Hence, for $s_0\gg1$ sufficiently large, $\phi (z) = \phi_g^{-1} \circ \phi_f (z): V_f(s_0) \to V_g (s_0) $ is an analytic isomorphism and $\phi=h$ on $\cA_f(s_0)$. Moreover, $\phi$ is an  analytic conjugacy between $f: V_f(s_0) \to V_f(s_0)$ and  $g: V_g(s_0) \to V_g(s_0)$.

We  assume that there exists $s_\star > 1$ such that $\phi$ extends to
$V_f({s_\star})$, that is, we have an analytic isomorphism $\phi: V_f ({s_\star}) \to
V_g({s_\star})$
extending the B\"ottcher coordinate change and conjugating $f: \cA_f(s_\star) \to \cA_f(s_\star)$ with  $g: \cA_g(s_\star) \to \cA_g(s_\star)$ such that 
$\phi(x)=h(x)$ for all $x \in \cA_f({s_\star})$.
To prove the proposition it is sufficient to
show that $\phi$ in fact extends to an analytic isomorphism
$\ol{\phi}: V_f (s) \to V_g (s)$ for some $s < s_\star$
such that $\overline{\phi}(x)=h(x)$ for all
$x \in \cA_f (s)$. It automatically follows that $\overline{\phi}$ is a conjugacy since $\overline{\phi}^{-1}\circ g \circ \overline{\phi} (z) = f(z)$ in an
open set. 

\smallskip
Consider $x_0 \in \partial V_f (s_\star)$. Keep in mind that there are only finitely many such $x_0$. Moreover, $]x_0, \infty[ \subset V_f (s_\star)$ and
$\phi (]x_0,\infty[) = ]x_0',\infty[$ for some
$x_0' \in \partial V_g (s_\star)$.
The main step is to extend $\phi$ to  a neighborhood of $x_0$ in $\aoneberk$.
We consider three cases.

\emph{Case 1. $\deg_{x_0} f =1$.}
Consider $ x \succ x_0$ close to $x_0$ such that $x$ is not in the ramification tree of $f$.
Set $x' := \phi (x)$.
Then $\deg_x f =1$ and, therefore, $\deg_{x'} g =1$, since $\phi$ is a conjugacy in $V_f(s_\star)$.
Let $U_0$ (resp. $V_0$) be 
the Berkovich open disk formed by points in the direction
of $x_0$ (resp. $x'_0$) at $x$ (resp. $x'$).
We may  denote by $G$ the inverse of $g: V_0 \to g(V_0)$.
It follows that $\ol{\phi}:=G \circ \phi \circ f : U_0 \cap V_f (s_\star^{1/d}) \to
V_0 \cap V_g (s_\star^{1/d})$ is a well defined analytic isomorphism that
coincides with $\phi$ in  $U_0 \cap V_f (s_\star)$ and 
agrees with $h$ on the points of $\cA_f$ contained in 
$U_0 \cap V_f (s_\star^{1/d})$.

\emph{Case 2. $\deg_{x_0} f \ge 2$ and $x_0$ is a type II point.}
In particular, $x_0$ lies in the ramification tree and therefore in $\cA_f$. 
Our assumption that $\phi = h$ on $\cA_f(s_\star)$
guarantees that $\phi$ maps $V'_f(s_\star)$ onto $V'_g(s_\star)$.
The B\"ottcher coordinate change
$\phi$ is asymptotic to the identity at infinity,
therefore, $\phi(x)=x$ and $T_{x}\phi=\operatorname{id}$ 
for all $x\in\cA_f(s_\star)$ sufficiently close to $\infty$.
By Proposition~\ref{p:iso-adm} and
Lemma~\ref{l:unique-admissible} we conclude that
$T_x \phi = T_x h$
for all $x \in \cA_f(s_\star) \subset \sk V'_f(s_\star)$.

% ,
% there exists a translation $ \tau_x(z)=z+b_x$ such that
% $\phi(y)=\tau_x(y)$ and $T_y \phi = T_y \tau_x$
% for all $y \in \cA_f(s_\star)$ close to $x$.

% , $h=\phi$ is an isometry, it follows that $\phi:V'_f(s_\star)\to V'_g(s_\star)$ is an analytic isomorphism. Noting that for $x\in\cA_f(s_\star)$ with $|x|\gg 1$, $\phi(x)=x$ and $T_{x}\phi=id$,  by Proposition~\ref{p:iso-adm}, we conclude that for each $x \in \cA_f(s_\star)$ , there exists a neighborhood $V_x\subset V_f(s_\star)$ of $x$ and $ \tau_x(z)=z+b_x$ for some $b_x\in\field$ such that for all $y\in V_x$, $\phi(y)=\tau_x(y)$ and $T_y \phi = T_y \tau_x$. 
% Observe th
% at $T_x \varphi = T_x \tau_x$ for all $x \in \cA_f(s_\star)$ by Proposition~\ref{p:iso-adm}.\footnote{massage this proposition so that everything fits}
Let $x_1=f(x_0)$ and consider translations $\tau_0, \tau_1$ that locally agree
with $h$ around $x_0, x_1$, respectively. After checking that (1)--(4) of Lemma~\ref{l:t-criteriumA} hold for $\psi_1=g \circ \tau_0$ and $\psi_2=\tau_1 \circ f$, we will conclude that 
$T_{x_0} (g \circ \tau_0) = T_{x_0} (\tau_1 \circ f)$. 
Tameness of $f$ and $g$ guarantees that  $p\nmid\deg_{x_0} (g \circ \tau_0)$ and $p\nmid\deg_{x_0} (\tau_1 \circ f)$, and hence (1) of  Lemma~\ref{l:t-criteriumA} follows.
Now for (2), pick $\vec{v}\in T_{x_0}\aoneberk$ and let $\vec{u}:=T_{x_0} \tau_0(\vec{v})$. We must show that $\deg_{\vec{u}} g = \deg_{\vec{v}} f$.
Since $h$ agrees with $\tau_0$ on a neigborhood of $x_0$, the direcion
$D(\vec{v})$ is disjoint from $\cA_f$ if and only if  the direction $D(\vec{u})$ is disjoint from $\cA_g$.
When $D(\vec{v})$ and $D(\vec{u})$ are disjoint from $\cA_f$ and $\cA_g$, respectively, $\deg_{\vec{v}}f=1=\deg_{\vec{u}}g$.
When $D(\vec{v})$  intersects $\cA_f$, consider a small arc $]x,x_0] \subset \cA_f$ in the direction $D(\vec{v})$. Then the hyperbolic length of
$]x,x_0]$ under $h \circ f$ (resp. $g \circ h$) is multiplied by a factor
of  $\deg_{\vec{v}}f$ (resp. $\deg_{\vec{u}}g$), since $h$ is an isometry.
But $h$ is also  a conjugacy, so these factors must agree. That is,  $\deg_{\vec{v}}f = \deg_{\vec{u}}g$.
Thus   (2) of  Lemma~\ref{l:t-criteriumA} holds.
For $x \in]x_0,\infty[ \subset V_f(s_\star)$ sufficiently close to $x_0$,
we have $T_x (\phi \circ f) = T_x (g \circ \phi)$. Therefore, 
$T_x (\tau_1 \circ f) = T_x (g \circ \tau_0)$ since $T_x\phi = T_x h$ and
$T_x h = T_x \tau_i$ for $x \in ]x_i,\infty[$. 
% setting $y_1:=f(y_0)$, we have that  $\psi_1(y_0)=g \circ h(y_0)=h \circ f(y_0)=\psi_2(y_0)$ together with $T_{y_0}\tau_0=T_{y_0}\tau_{y_0}=T_{y_0}\phi$ and $T_{y_1}\tau_1=T_{y_1}\tau_{y_1}=T_{y_1}\phi$ since $\tau_0=h=\phi=\tau_{y_0}$ in a small neighborhood of $y_0$ in $\cA_f$ and $\tau_1=h=\phi=\tau_{y_1}$  in a small neighborhood of $y_1$ in $\cA_f$, so $T_{y_0}\psi_1=T_{y_0}(g \circ\tau_0)=T_{h(y_0)}gT_{y_0}\tau_0=T_{\phi(y_0)}gT_{y_0}\phi=T_{y_0}(g\circ\phi)=T_{y_0}(\phi\circ f)=T_{y_1}\phi T_{y_0}f=T_{y_1}\tau_1T_{y_0}f=T_{y_0}(\psi_2)$, and hence  the assumption
That is, (3) of  Lemma~\ref{l:t-criteriumA} holds.
Finally (4) also holds since $\cA_f$ has points
in at least one bounded direction at $x_0$. From Lemma~\ref{l:t-criteriumA}
now we have $T_{x_0} (g \circ \tau_0) = T_{x_0} (\tau_1 \circ f)$.
% In a neighborhood of $x_0$ in $\cA_f$, we have $\psi_1=g\circ h=h\circ f=\psi_2$. It follows that for any direction $\vec{v}\in T_{x_0}\aoneberk$ with $D(\vec{v})\cap\cA_f\not=\emptyset$, we have $T_{\vec{v}}\psi_1=T_{\vec{v}}\psi_2$, moreover, since $x_0\in\cA_f$, we can choose the above $\vec{v}\not=\vec{v}_{x_0}(\infty)$, so  the assumption (4) in  Lemma~\ref{l:t-criteriumA} holds.

By  Lemma \ref{lem:lift} in the appendix, there exists
an analytic map $\psi$ defined in a neighborhood $V$ of $x_0$ such that $g \circ \tau_0 \circ \psi = \tau_1 \circ f$ in $V$ and 
$\psi(y_0) = y_0$, $T_{y_0} \psi = \operatorname{id}$ for $y_0\in V$ with $\dist (y_0,x_0)$ small enough.
For a small annulus $A\subset V$ with $\sk A=]x_0,x[$ for some $x_0\precneqq x\in V$ sufficiently close to $x_0$, considering the map $f:A\to f(A)$ and setting $\psi_1:=\phi^{-1}\circ \tau_0 \circ \psi$, we obtain that $\psi_1: A\to A$ satisfies $f\circ\psi_1=f\circ\phi^{-1}\circ \tau_0 \circ \psi=\tau_1^{-1}\circ g\circ \tau_0 \circ \psi=\tau_1^{-1}\circ \tau_1 \circ f=f$ in $A$, moreover, for any $y_0\in\sk A$, we have $\psi_1(y_0)=\phi^{-1}\circ \tau_0 \circ \psi(y_0)=\phi^{-1}\circ \tau_0(y_0)=h^{-1}\circ h(y_0)=y_0$ and $T_{y_0}\psi_1=T_{y_0}(\phi^{-1}\circ \tau_0 \circ \psi)=T_{y_0}(\phi^{-1}\circ \tau_0)=T_{h(y_0)}(\tau_{y_0}^{-1})T_{y_0}\tau_0=\operatorname{id}$. Then by Corollary~\ref{c:identity}, we conclude that $\psi_1=\operatorname{id}$ and hence $\phi = \tau_0 \circ \psi$ in $A$. Thus, $\phi$ extends to a neighborhood of $x_0$ and the extension $\ol{\phi}:=\tau_0 \circ \psi$ agrees with $\tau_0=h$ on the points of $\cA_f$ sufficiently close to $x_0$.

\emph{Case 3. $\deg_{x_0} f \ge 2$ and $x_0$ is a type III point.} Let $A$ be a small annulus such that $x_0\in A$ and $\sk A\subset\cA_f$. Denote by $A_1:=\phi\circ f(A)$ and let $A'$ be the connected component of $g^{-1}(A_1)$ containing $h(x_0)$. Applying Corollary~\ref{c:extension-annulus} to $\phi\circ f: A\to A_1$ and $g: A'\to A_1$, we obtain a unique analytic isomorphism $\ol{\phi}: A\to A'$ that is the extension of $\phi:A\cap V_f (s_\star)\to A'\cap V_g (s_\star)$ satisfying $g\circ\ol{\phi}=\phi\circ f$ in $A$. Observe that  $\ol{\phi}=h$ in $\sk A$ since $h(x_0) = \ol{\phi}(x_0)$ and $\ol{\phi}$ acts as an isometry on $\sk A$.

Thus for $x_0 \in \partial V_f (s_\star)$, we obtain an extension $\ol{\phi}$ of $\phi$ in a neighborhood of $x_0$ such that $\ol{\phi}=h$ on the points of $\cA_f$ in this neighborhood.
\end{proof}

\begin{proof}[Proof of Theorem~\ref{ithr:D}]
Theorem~\ref{ithr:D} is an immediate consequence of Corollary \ref{c:ithr-D-easy} and Proposition \ref{p:ithr-D-hard}.
\end{proof}

\section{Existence of extendable conjugacies}
\label{sec:existence}

The goal of this section is to prove Propositions~\ref{prop:close} and \ref{prop:p-to-a} which together become a slightly stronger version of Theorem \ref{ithr:C}. 
%with an stratification of  the restricted post-ramification trees for tame polynomials. %with the aim of passive critical points.% In Section \ref{sec:tree}, we equip the extended trees introduced in Section \ref{s:extended-tree} with a natural graph-theoretic combinatorial structure. In Section \ref{sec:passive}, we define the passive/active critical points. Then in Section \ref{s:assumptions}, we discuss the hypotheses of Theorem \ref{ithr:C}, and  in Section \ref{sec:existence}, we construct an admissible conjugacy defined in Section \ref{s:admissible-definitions}. Finally, applying Theorem \ref{ithr:D} and a classical analytic removability result, we prove Theorem \ref{ithr:C} in Section \ref{s:proof-C}.

\subsection{Trimmed dynamical core}
\label{sec:trimmed}
Given a tame polynomial $f {\in\Polyd}$, recall that the dynamical core is
$$\cA_f = \cB(f) \cap \bigcup_{c \in \crit(f),  n \ge 0} ]f^{\circ n} (c), \infty[.$$

 Sometimes
 it will be convenient to consider the subtree of $\cA_f$
 formed by the elements
that escape to infinity through a hyperbolic $\rho$-neighborhood of the axis $]x_f,\infty[$ for some $\rho >0$ (possibly $\infty$).
Note that since $f$ is tame, {in $|z| > |x_f|$}
the distance to the axis is preserved by
dynamics. More precisely, 
for all $x \in \H$ such that $|x| > |x_f|$ we have that
$$\dist(x,]x_f,\infty[) = \dist(f(x),]x_f,\infty[).$$

\begin{definition}
  \label{d:trimmed}
  Consider a tame polynomial  $f {\in\Polyd}$.
  Given $\rho \in ]0,\infty]$, we say that the \emph{$\rho$-trimmed dynamical core of $f$}, denoted by $\cT_f$, is the set formed by all $x \in \cA_f$ such that
$\dist (f^{\circ n} (x), ]x_f,\infty[) < \rho$  for all $n$ sufficiently large.
\end{definition}

Keep in mind that $\cA_f$ is the $\rho$-trimmed dynamical core of $f$ with $\rho=\infty$. {Also, from the previous discussion, $x \in \cA_f$ with $|x|>|x_f|$
lies in the $\rho$-trimmed dynamical core $\cT_f$ if and only if
$\dist(x,]x_f,\infty[) < \rho$, equivalently, $\log (|x|/\diam(x)) < \rho$.} 

\begin{lemma}
  Let $f \in \Polyd$ be a tame polynomial. For any $\rho \in ]0,\infty]$, the
  corresponding $\rho$-trimmed dynamical core $\cT_f$ is a locally finite open
  subtree of $\H$ which is forward invariant (i.e. $f (\cT_f) \subset \cT_f$).
  Moreover, if $x \in \cT_f$, then $]x,\infty[ \subset \cT_f$. Furthermore,
  if a critical point $c$ lies in
  a direction $D_x(\vec{v})$ at some $x \in \cT_f$, then $D_x(\vec{v}) \cap \cT_f \neq \emptyset$.
\end{lemma}

\begin{proof}
  Directly from its definition we conclude that $\cT_f$ is forward invariant. 
  Moreover,  $\cT_f$ is obtained from $\cA_f$ by cutting off some of its ``branches''.
  Indeed, 
  if $x \in \cA_f \setminus \cT_f$ and $y \prec x$, then $f^{\circ n}(y) \prec f^{\circ n}(x)$ for all $n \ge 0$. Hence, for $n$ sufficiently large $f^{\circ n}(y)$
  is at distance at least $\rho$ from the axis and therefore $y \notin \cT_f$. It follows that $\cT_f$ is connected and relatively open in $\cA_f$. Therefore
  $\cT_f$ is also a locally finite open subtree of $\H$.
 If a direction $D_x(\vec{v})$ contains a critical point $c$ for some
 $x \in \cT_f$, then $D_x(\vec{v}) \cap \cT_f \neq \emptyset$, since
 $\cT_f$ is open in $\cA_f$ and $]c,\infty[ \cap \cB(f) \subset \cA_f$.  
\end{proof}

Given a $\rho$-trimmed dynamical core $\cT_f$,
we say that $x \in \cT_f$  is a \emph{vertex} of $\cT_f$ if $f^{\circ n }(x)$
is a branch point of  $\cT_f$ for some $n \ge 0$. The set formed by the vertices of $\cT_f$ is denoted by $\cV_f$. An \emph{edge} of $\cT_f$ is a connected component of $\cT_f \setminus \cV_f$. 

\begin{proposition}[Vertices of $\cT_f$]
  \label{p:vertices-a}
  Consider a non-simple tame polynomial $f\in\Polyd$ with base point $x_f$ {and $\rho \in ]0,\infty]$}.
  Let $\cV_f$ be the set of vertices of {the} $\rho$-trimmed dynamical core $\cT_f$.
  Then $x_f \in \cV_f$ and $f(\cV_f) \subset \cV_f$. Moreover, 
  the set of accumulation points  of $\cV_f$ in $\aoneberk$ is contained
  in the Julia set $\cJ(f)$.
\end{proposition}

\begin{proof}
  We claim that if $x \in \cT_f$ is a branch point of $\cT_f$ then $f(x)$ is also a branch point. Indeed,
  let us proceed by contradiction assuming that $x$ is branch point and $f(x)$ is not. Then all bounded directions at $x$ containing elements of $\cT_f$
  are contained in the preimage of a single direction $\vec{w}$ at $f(x)$.
  By the previous lemma,  $\vec{w}$ is the unique critical value direction of
  $T_xf$ and therefore $T_xf$ is a monomial. Hence
  the preimage of $\vec{w}$ is a singleton and $x$ is not a branch point.

  In view of the previous paragraph, if $f^{\circ n} (x)$ is a branch point of $\cT_f$, then $f^{\circ n}(f(x))$ is a branch point. Thus, $f (\cV_f) \subset \cV_f$.

  Forward invariance of the  axis $[x_f,\infty[$ implies that $x_f \in \cT_f$. 
  Since $f$ is nonsimple and tame we may apply \eqref{eq:diameter} to conclude that there exists a critical value $v$
  of $f$ such that $|v| = |f(x_f)|$.  It follows that $f(x_f)$ is a branch point of $\cA_f$. Therefore, $x_f \in \cV_f$.

  Let us say that two escaping critical points $c, c'$
  are eventually in the same direction  if there exists
  $m, n \ge 0$ such that $|x_f|< |f^{\circ n}(c)| = |f^{\circ m}(c')| = R$
  and $|f^{\circ n}(c) - f^{\circ m}(c')|< R$ (i.e. they lie in the same direction 
  at $x_{0,R}$).
  There exists $N \ge 0$ such that for all pairs of critical points
  $c, c'$ which   are eventually in the same direction, then $n,m$ above can be chosen to be at most $N$.
  Now consider $X : = \{ x : |f^{\circ N}(x_f)| < |x| \le |f^{\circ N+1}(x_f)| \}$.
It follows that $\cA_f \cap X$ maps homoemorphically onto its image under $f$. Moreover, $\cA_f \cap X$ is the union of finitely many arcs of the form
$]z,f^{\circ N+1}(x_f)]$ with the arc $]f^{\circ N}(x_f),f^{\circ N+1}(x_f)]$.
In particular, $\cA_f \cap X$ contains finitely many branch points.
  Also, $X$ is a ``fundamental domain'' for the action of $f$ on $\cB(f)$.
  Namely, $\cB(f)$ is the disjoint union of the sets $f^{\circ n}(X)$ for $n \in \Z$. 
   Since every branch point of $\cT_f$ is the iterated image or preimage of one in $\cA_f \cap X$, the {conclusion} follows. 
\end{proof}

Now we introduce an increasing collection $\{\cT_f^{(j)}\}_{j \ge 0}$
of subtrees of $\cT_f$ that saturate $\cT_f$.
Set $\cT_f^{(0)}:= \cT_f \cap \{ x : |x| > |x_f| \}$.
Recursively, the level $j+1$ subtree  $\cT_f^{(j+1)}$ is obtained by adjoining to $\cT_f^{(j)}$
all the elements $x \in \cT_f$ such that there exists $y \in \cT_f^{(j)}$
with the property that $]x,y[$ contains at most one vertex of $\cT_f$.

\begin{corollary}
  Consider a nonsimple tame polynomial $f\in\Polyd$ and $\rho \in ]0,\infty]$.
  Let $\cT_f$ be the $\rho$-trimmed
  dynamical core of $f$.
  Then $$\cT_f = \bigcup_{j \ge 0} \cT_f^{(j)}.$$
  Moreover,
  $$f (\cT_f^{(j+1)}) \subset \cT_f^{(j)},$$
  for all $j \ge 0$.
\end{corollary}

\begin{proof}
  For all $x \in \cT_f$, observe that $]x,x_f[ \cap \cV_f$ is finite, since $[x,x_f]$ is free of accumulation points of $\cV_f$,
  by Proposition~\ref{p:vertices-a}.

  Given $x \notin  \cT_f^{(0)}$ and $j \ge 1$, note that $x \in \cT_f^{(j)}$
  if and only if  $]x,x_f]$ contains at most $j$ elements of $\cV_f$.
  By the previous paragraph, we have that every element of $\cT_f$ lies in
  some $\cT_f^{(j)}$.

  By definition, $f( \cT_f^{(0)}) \subset  \cT_f^{(0)}$.
  For all $j \ge 1$ and  $x \in \cT_f^{(j)} \setminus \cT_f^{(0)}$, our map
  $f$ acts bijectively on $]x,x_f]$ and therefore
  $]f(x),f(x_f)]$ contains
  at most $j$ elements of $\cV_f$. 
  If $f (x) \notin \cT_f^{(0)}$,
  then  $]f(x), f(x_f)]$ is the disjoint union of $]f(x), x_f]$
  and $]x_f, f(x_f)]$. Hence, $]f(x), x_f]$ contains at most $j-1$ elements of $\cV_f$ and $f(x) \in \cT_f^{(j-1)}$.
  It follows that  $f( \cT_f^{(j)}) \subset  \cT_f^{(j-1)}$.
\end{proof}

\begin{proposition}[Edges of $\cT_f$]
  \label{l:edges-a}
  Consider a nonsimple tame polynomial $f$ and $\rho \in ]0,\infty]$. Denote by $\cT_f$ the $\rho$-trimmed
  dynamical core of $f$.
  Let $e$ be an edge of $\cT_f$. Then there exist $a \in \cV_f \cup \partial \cT_f$ and $b \in \cV_f$ 
  such that  $a \prec b$ and $e = ]a,b[$.
  Moreover,
  $f(e) = ]f(a),f(b)[$ is also an edge of $\cT_f$
  and $\deg_x f = \deg_a f$, for all
  $x \in e$. Furthermore, for all $y, y' \in \overline{e}$,
  $$\dist (f(y),f(y')) = \deg_af \cdot \dist (y,y').$$
\end{proposition}

\begin{proof}
  An edge $e$ is, by definition, a connected component of $\cT_f \setminus \cV_f$. If the extreme points are $a$ and $b$, then $]a,\infty[$ and $]b,\infty[$ are contained in $\cT_f$ and, $]a,b[$ is branch point free. It follows that $a \prec b$ or $b \prec a$. Without loss of generality we assume that $a \prec b$.
    By definition if a preimage $w$ of a vertex lies in $\cT_f$, then it is a vertex. Therefore, $f(e)$ is an edge. For all $x \in \H$, since $f$ is tame,  $\deg_x f -1$ agrees with the number of critical points in the disk $\ol{D}_x$,
    counted with multiplicities {(see Lemma \ref{lem:RH})}. This number is constant along edges and therefore $\deg_x f = \deg_af$ for all $x \in e$.
    Thus, the hyperbolic metric is expanded by a factor $\deg_af$ along $e$
    {(see Lemma \ref{lem:factor})}.
\end{proof}

\subsection{Close B\"ottcher coordinates of escaping critical points}
\label{sec:close}

We start by introducing a notion that quantifies how close are the B\"ottcher coordinates of escaping critical points with the agreement that $\exp(-\infty):=0$.

\begin{definition}
  \label{def:close}
  Consider two tame critically marked
polynomials $f$ and
$g$ in $\Polyd$ with the same base point $x_0$.
Let $\phi_f, \phi_g : \aoneberk \setminus \overline{D}(a,|x_0|) \to \aoneberk \setminus \overline{D}(a,|x_0|)$ be the
corresponding B\"ottcher coordinates.
  Given $\rho \in ]0,\infty]$ we say that the \emph{escaping critical points of $f$ and $g$ have $\rho$-close 
    B\"ottcher coordinates} if $|f^{\circ m}(c_i(f))| > |x_0|$  if and only if $|g^{\circ m}(c_i(g))| > |x_0|$  and in this case:
 % \begin{enumerate}
%  \item  $|\phi_f|(c_i(f))=|\phi_g|(c_i(g))$.
%  \item
  \begin{equation}
    \label{eq:rho}
  |\phi_f(f^{\circ m}(c_i(f)))- \phi_g(g^{\circ m}(c_i(g)))| \le
  \exp(-\rho) |f^{\circ m}(c_i(f))|.
\end{equation}  
 % \end{enumerate}
\end{definition}

% \begin{definition}
% We say that the \emph{escaping critical points of $f$ and $g$ have $\rho$-close 
%   B\"ottcher coordinates} if the following holds:
% \begin{enumerate}
% \item[(i)] $\varphi (\cT^{(0)}_f) = \cT^{(0)}_g$ where $\cT^{(0)}_f$ and $\cT^{(0)}_g$ are the level $0$ subsets of the corresponding $\rho$-trimmed dynamical cores.

% \item[(ii)] Let $\pi_f: \aoneberk \to \overline{\cT}^{(0)}_f$ and $\pi_g: \aoneberk \to \overline{\cT}^{(0)}_g$ be the canonical retractions. For all $i=1, \dots, d-1$
% and all $n \ge 0$, we have
% $$\varphi \circ \pi_f(f^{\circ n}(c_i(f))) = \pi_g(g^{\circ n}(c_i(g))).$$
% \end{enumerate}

% Condition (ii) is equivalent to the following:

% (ii') For all $r >0$, if $z = f^{\circ n}(c_i(f))$ and $w = g^{\circ n}(c_i(g))$
% for some $n \ge 0$, $1 \le i < d$, then:
% $x_{z,r} \in \cT^{(0)}_f$ if and only if $x_{w,r} \in \cT^{(0)}_f$ and, in this case, $\varphi(x_{z,r}) = x_{w,r}$.
% \end{definition}

{For any} passive critically marked analytic family of tame polynomials, observe that escaping critical points have $\infty$-close
B\"ottcher coordinates if and only if they have
constant  B\"ottcher coordinates, by
Definition~\ref{def:constant-B}.

Note that if $f$ and $g$ have $\rho$-close B\"ottcher coordinates and \eqref{eq:rho} holds for $i$ and $m$, then
$$|f^{\circ m}(c_i(f))|= |\phi_f(f^{\circ m}(c_i(f)))| = |\phi_g(g^{\circ m}(c_i(g)))|=|g^{\circ m}(c_i(g))|.$$
For every escaping critical point $c_i(f)$ let  $m_i \ge 1$ be  such that $ |x_0|<|f^{\circ m_i}(c_i(f))| \le |x_0|^d$, equivalently $m_i$ is the smallest integer such that $ |x_0|<|f^{\circ m_i}(c_i(f))|$. If for all such escaping $c_i(f)$
we have that  $|g^{\circ m_i}(c_i(g))| > |x_0|$ and
\eqref{eq:rho} holds for $m=m_i$, then $f$ and $g$ have $\rho$-close B\"ottcher coordinates,  since $z \mapsto z^d$ is a $\dist$-isometry on
$D_{x_G}(\infty) \setminus \, ]x_G,\infty[$.

\begin{proposition}\label{prop:close}
  Let $\{f_\lambda \}$ be an analytic family of critically marked tame polynomials in $\Poly_d$ parametrized by an open disk $\Lambda \subset \field$
  with radius in $|\field^\times|$.
  Assume that all the critical points are passive in $\Lambda$. Then 
  for every closed disk $\overline{\Omega}$ contained in $\Lambda$,
  there exists $\rho>0$ such that for all
  $\lambda_1, \lambda_2 \in \ol{\Omega} $, the escaping critical points
  of $f_{\lambda_1}$ and $f_{\lambda_2}$ have $\rho$-close B\"ottcher coordinates. 
\end{proposition}

\begin{proof}
 % \jk{\tiny We may assume that $f_\lambda$ is not simple for all $\lambda$. Consider a parameter $\lambda_0$ and, by .., there exists an escaping critical point $c(\lambda_0)$ such that $|f_{\lambda_0}(c(\lambda_0))| = |f_{\lambda_0} (x_{\lambda_0})| = |x_{\lambda_0}|^d$. For all $\lambda \in \Lambda$, we must have $|f_{\lambda}(c(\lambda))|  = |x_{\lambda_0}|^d$, for otherwise $\lambda \mapsto f_{\lambda}(c(\lambda)) - c(\lambda)$ would have a zero in $\Lambda$ contradicting that $c(\lambda)$ lies in the basin of infinity for all $\lambda$. It follows that $|x_{\lambda_0}| \le |x_\lambda|$ for all $\lambda$. By symmetry we have that the basepoint of $f_\lambda$ is independent of $\lambda$.}

  From Lemma~\ref{l:passive-constant} we know that the base point $x_\Lambda$ is independent of $\lambda \in \Lambda$.
    Denote by $r_\Lambda$ the diameter of $x_\Lambda$.
    Consider a closed disk $\ol{\Omega} \subset \Lambda$ and pick a reference parameter $\lambda_1 \in \ol{\Omega}$.
    Assume that $|f_{\lambda_1}^{\circ n} (c_i(\lambda_1))| > r_\Lambda$ where $n$ is the smallest number with this property. To simplify notation let $\alpha_i(\lambda) = f_{\lambda}^{\circ n} (c_i(\lambda))$. Then $|\alpha_i(\lambda)|$  takes a constant value, say $r_i$, because otherwise $\alpha_i(\lambda)-c_i(\lambda)$ would vanish at some parameter. 
    Moreover, for the same reason, the direction of $\alpha_i(\lambda)$ at $x_{0,r_i}$ is also constant in $\Lambda$. Since the tangent map of $\phi_\lambda$
    at  $x_{0,r_i}$ is the identity, in standard coordinates, 
    $$|\phi_{\lambda} (\alpha(\lambda))-\phi_{\lambda_1} (\alpha(\lambda_1))|
    < r_i$$
    for all $\lambda \in \Lambda$.
    Since $\phi_\lambda(\alpha_i(\lambda))$ is analytic, there exists
    $\rho >0$ such that for all $\lambda_2 \in \overline{\Omega}$ and all $i$:  
     \begin{equation*}
      %\label{eq:rho}
      |\phi_{\lambda_2} (\alpha_i(\lambda_2))-\phi_{\lambda_1} (\alpha_i(\lambda_1))| \leq \exp(-\rho) \cdot r_i.
    \end{equation*}

\end{proof}

\subsection{From passive critical points to extendable conjugacies}
\label{sec:passive2extendable}

To prove Theorem \ref{ithr:C},  we first show the following result.

\begin{proposition}
    \label{prop:p-to-a}
  Let $\{f_\lambda \}$ be a passive analytic family of critically marked tame polynomials in $\Poly_d$ parametrized by a disk $\Lambda \subset \field$. Consider $\rho >0$ and denote by $\cT_\lambda$ the $\rho$-trimmed dynamical core of $f_\lambda$. Also  let
  $$\cC := \{ \alpha: \Lambda \to \aone : \alpha(\lambda) = f^\ell_\lambda (c_i(\lambda)) \text{ for some } \ell \ge 0, 1 \le i < d \}. $$
  
Suppose that the B\"ottcher coordinates of escaping critical points of $f_{\lambda_0}$ and $f_\lambda$ are $\rho$-close for all $\lambda \in \Lambda$.
Then there exists a unique map  $h_\lambda : \cT_{\lambda_0} \to \cT_\lambda$ such that
  if $x = x_{\alpha(\lambda_0), r} \in \cT_{\lambda_0}$ for some $\alpha \in \cC$ and $r >0$, then
  $$h_\lambda (x) = x_{\alpha(\lambda), r}.$$
  Moreover, $h_\lambda$ is an extendable conjugacy between $f_{\lambda_0}:\cT_{\lambda_0}\to\cT_{\lambda_0}$ and $f_{\lambda}:\cT_{\lambda}\to\cT_{\lambda}$. %extending B\"ottcher coordinates and respecting critical markings.
\end{proposition}

\begin{proof}
  In view of Lemma~\ref{l:passive-constant}, denote by $x_\Lambda$ the base point of $f_\lambda$, which is independent of $\lambda \in \Lambda$.

  Observe that for any $\lambda\in\Lambda$, every point in $\cT_\lambda$ has form $x_{\alpha(\lambda),r}$ for some $\alpha\in\cC$ and $r>0$. So if $h_\lambda$ exists, it is unique.

  The existence of $h_\lambda$ will follow once we prove 
  by induction on $j \ge 0$ the following assertions:

  (1) if $\alpha, \beta \in \cC$ and $r >0$ are such that $x_{\alpha(\lambda_0),r} = x_{\beta(\lambda_0),r} \in \cT_{\lambda_0}^{(j)}$, then $x_{\alpha(\lambda),r} = x_{\beta(\lambda),r} \in  \cT_{\lambda}^{(j)}$ for all $\lambda \in \Lambda$.
  In this case, we let  $h_\lambda (x_{\alpha(\lambda_0),r}) := x_{\alpha(\lambda),r}$.

  (2)  $h_\lambda:  \cT_{\lambda_0}^{(j)} \to \cT_{\lambda}^{(j)}$ is an extendable
  conjugacy.

  \medskip
  To simplify notation we set $\lambda_0:=0$ and employ subscripts accordingly.

  \bigskip
  We start with the case in which $j=0$.
   Consider $\alpha \in \cC$ such that, $|\alpha(0)| > |x_\Lambda|$.
  Set $r_\alpha := |\phi_0|(\alpha(0))$ and $R_\alpha := \exp(-\rho) r_\alpha$.
  Since the
    B\"ottcher coordinates of
    escaping critical points are $\rho$-close, it follows that
    $r_\alpha=|\phi_\lambda|(\alpha(\lambda)) = |\alpha(\lambda)|$ for all $\lambda$.
    The inequality
    $$|\phi_\lambda(\alpha(\lambda)) - \phi_0(\alpha(0))| \le R_\alpha$$
    which by definition holds for all $\alpha$ such that $|x_\Lambda| < |\alpha(0)| \le |x_\Lambda|^d$, extends to all $\alpha \in \cC$
    such that $|\alpha(0)|>|x_\Lambda|$. This follows since $\phi_\lambda$ and $z \mapsto z^d$
    are isometries in the hyperbolic metric on connected components of the complement of $]x_\Lambda,\infty[$ not containing $z=0$.
    Note that $x_{\alpha(\lambda),r} \in \cT^{(0)}_\lambda $
    if and only if  $r > R_\alpha$. In this case,
    $\phi_0(x_{\alpha(0),r})= \phi_\lambda(x_{\alpha(\lambda),r})$ since B\"ottcher coordinates are diameter preserving.
    Therefore,  $h_\lambda := \phi_\lambda^{-1} \circ \phi_0 : \cT^{(0)}_0 \to \cT^{(0)}_\lambda$ is well defined and (1) holds for all $\alpha, \beta$ such that
    $|\alpha(0)| > |x_\Lambda|$ and $|\beta(0)|> |x_\Lambda|$. Statement (1) extends to arbitrary $\alpha, \beta \in \cC$ since in the rest of the cases
    $x_{\alpha(\lambda),r} = x_{0,r}$ for some $r > |x_\Lambda|$ whenever $x_{\alpha(\lambda),r} \in \cT^{(0)}_\lambda $.
    Moreover, 
    the open set with skeleton $\cT^{(0)}_0$ is mapped by $\phi_\lambda^{-1} \circ \phi_0$ onto the open set
    with skeleton $\cT^{(0)}_\lambda$. Therefore, by Proposition~\ref{p:iso-adm}, 
    $h_\lambda$
    is an extendable conjugacy and (2) also holds for $j=0$.

 %  $h_\lambda : \cA_{\lambda_0}^{(j)} \to \cA_{\lambda}^{(j)}$ 
% The B\"ottcher coordinates $\phi_{\lambda}^{-1}\circ\phi_{\lambda_0}$ gives $h_\lambda : \cA_{\lambda_0}^{(0)}
%   \to  \cA_\lambda^{(0)}$. In fact, $\phi_{\lambda}^{-1}\circ\phi_{\lambda_0}$ is diameter preserving and sends
%   postcritical orbits to postcritical orbits by the assumption on the B\"ottcher coordinates. Moreover, for $x\in\cA_{\lambda_0}^{(0)}$, setting $y:=f_{\lambda_0}(x)$ and $x'=h_{\lambda}(x)$ and recalling Section \ref{sec:conjugacy} for the definition of the tangent map of $h_\lambda$, we have $T_yh_{\lambda}\circ T_xf_{\lambda_0}=T_{x'}f_{\lambda}\circ T_xh_\lambda$.

 % Suppose that for some $j\ge 0$ and all $\lambda\in\Lambda$, $h_\lambda : \cA_{\lambda_0}^{(j)} \to \cA_\lambda^{(j)}$ is a well defined extendable conjugacy such that for any $x\in\cA_{\lambda_0}^{(j)}$, $T_yh_{\lambda}\circ T_xf_{\lambda_0}=T_{x'}f_{\lambda}\circ T_xh_\lambda$, where $y=f_{\lambda_0}(x)$ and $x'=h_{\lambda}(x)$. Then by continuity, $h_\lambda$ extends to $\partial\cA_{\lambda_0}^{(j)}$.  In what follows, we show $h_\lambda : \cA_{\lambda_0}^{(j+1)} \to \cA_\lambda^{(j+1)}$ is a well defined extendable conjugacy whose tangent map conjugates the tangent maps of $f_{\lambda_0}$ and $f_\lambda$.

  \medskip
    Now we assume that assertions (1) and (2) hold for some $j \ge 0$ and proceed
    to establish their validity for $j+1$.

    Consider an endpoint  $x_0 \in \partial\cT_{0}^{(j)}$ such that $x_0 \in \cT_{0}^{(j+1)}$ and 
    denote by $r_0$
    its diameter. We will establish assertions (1) and (2) for $x_0$ and the points on the edges of  $\cT_{0}^{(j+1)}$ growing from $x_0$.

    \smallskip
A key observation will be the following. 
      Assume that $\alpha, \beta \in \cC$ are such that
    $\alpha(0)$ and $\beta(0)$ lie in the same bounded direction at $x_0$ (i.e.,  $\alpha(0) \prec x_0$, $\beta(0) \prec x_0$ and
    $|\alpha (0) - \beta(0)| <r_0$). Then
    \begin{equation}
      \label{eq:direction}
      |\alpha(\lambda) - \beta(\lambda)| < r_0 \text{ for all }\lambda \in \Lambda.
    \end{equation}
    Indeed, $x_{\alpha(\lambda),r} = x_{\beta(\lambda),r} \in \cT_{\lambda}^{(j)}$ for all $r >r_0$. Thus,  $|\alpha(\lambda) - \beta(\lambda)| \le r_0$ for all $\lambda \in \Lambda$.
    By the Maximum Principle, the last inequality is strict, as claimed.

     Now, let us suppose that there exists
    a point $x_0' \prec x_0$ such that $e_0 = ]x_0',x_0[$ is an edge of $\cT_{0}^{(j+1)}$  which is not in $\cT_{0}^{(j)}$. Denote the diameter of $x_0'$ by $r_0'$. By construction, there exists
    $\alpha \in \cC$ such that $\alpha(0) \prec x_0'$. For all
    $\lambda \in \Lambda$, let $x_\lambda$ (resp. $x_\lambda'$)  be the point of diameter
    $r_0$ (resp. $r_0'$) such that $\alpha(\lambda) \prec x_\lambda' \prec x_\lambda$.

    Let us first prove that (1) holds if 
    $]x_\lambda',x_\lambda[$ is an edge of $\cT_{\lambda}^{(j+1)}$.
    Indeed, assuming that  $e_\lambda=]x_\lambda',x_\lambda[$ is an edge of $\cT_{\lambda}^{(j+1)}$,
    if $x_{\alpha(0),r} = x_{\beta(0), r} \in e_0$ for some $\beta \in \cC$, then $|\beta(\lambda) - \alpha(\lambda)| < r_0$ for all $\lambda \in \Lambda$, in view of \eqref{eq:direction}. Therefore, $|\beta(\lambda) - \alpha(\lambda)| \le r_0'$ since the edge $e_\lambda$ is contained
    in $]\beta(\lambda), x_\lambda[$. Thus, $x_{\alpha(\lambda),r} = x_{\beta(\lambda), r}$ because $r > r_0'$.

    To finish the proof of (1) for $j+1$ we must show that $e_\lambda$ is an edge of $\cT_{\lambda}^{(j+1)}$ for all $\lambda \in \Lambda$. Let $e_\lambda'= ]a_\lambda,x_\lambda[ \subset ]\alpha(\lambda), x_\lambda[$ be the edge of $\cT_{\lambda}^{(j+1)}$ in the direction of $\alpha(\lambda)$ at $x_\lambda$.
    Since $e_\lambda$ is an arc with constant hyperbolic length contained in $]\alpha(\lambda), x_\lambda[$ and $e'_0=e_0$,
    to show that $e'_\lambda = e_\lambda$ for all $\lambda \in \Lambda$ it suffices
    to stablish that the hyperbolic length of $e_\lambda'$ is also independent of $\lambda$. 
    Denote by $D$ the disk corresponding the direction  $\alpha(\lambda)$ at $x_\lambda$.
    Note that {by Lemma \ref{lem:RH}}, $f_\lambda : D \to f_\lambda(D)$ has degree $m+1$, where $m$ is the number of critical points in $D$, counted with multiplicity which  is independent of $\lambda$ in view of \eqref{eq:direction}.
  Observe that
    $f_\lambda (e_\lambda')$ is the edge of $\cT^{(j)}_\lambda$ contained
    in $]f_\lambda(\alpha(\lambda)), f_\lambda(x_\lambda)[$ with one endpoint at $f_\lambda(x_\lambda)$. By the inductive hypothesis (1), $h_\lambda$ is an isometry between $f_0 (e_0')$ and $f_\lambda (e_\lambda')$ and therefore the hyperbolic
    length of $f_\lambda (e_\lambda')$  is a constant independent of $\lambda$, say $L$.
    Thus, if $L_\lambda$ denotes the length of
    $e_\lambda'$ then $L_\lambda \cdot (m+1) = L$ {by Lemma \ref{lem:factor}}. Therefore, the length of
    $e_\lambda'$ is constant equal to the length of $e_0' =e_0$. It follows that $e_\lambda' = e_\lambda$ for all $\lambda$.

    \medskip
    To show that $h_\lambda : \cT^{(j+1)}_0 \to \cT_\lambda^{(j+1)}$ is an extendable conjugacy it only remains to show that it is locally a translation. For this, consider any $x \in \cT^{(j+1)}_0 \setminus \cT^{(j)}_0 $. Denote by $r$ the diameter of $x$. Let $\alpha \in \cC$ be such that $\alpha(0) \prec x$. For $\lambda \in \Lambda$, consider the traslation $$\tau_\lambda (z) := z + \alpha(\lambda) - \alpha(0).$$
    We claim that $\tau_\lambda$ agrees with $h_\lambda$ in a neighborhood of $x$.
    Let $\beta \in \cC$ be such that $\beta(0) \prec x$. Then, for all $\lambda \in \Lambda$,     $$|\beta (\lambda) - \alpha(\lambda)| \le r.$$
  Therefore, 
    $|\beta(\lambda) - \tau_\lambda (\beta(0))| \leq r$. But since  $\beta(0) = \tau_\lambda (\beta(0))$ we may apply the Maximum Principle to conclude that 
  $$|\beta(\lambda) - \tau_\lambda (\beta(0))| < r$$
  for all $\lambda \in \Lambda$. Thus, $\tau_\lambda$ maps the direction of $\beta(0)$ at $x$,
  to the direction of $\beta(\lambda)$ at $\tau_\lambda(x)$. Hence, it coincides with
  $h_\lambda$ in a small segment with one endpoint at $x$ contained in $]\beta(0),x[$. Since
  this occurs for all $\beta \in \cC$, the claim follows concluding the proof of the proposition. 
  \end{proof}

Theorem \ref{ithr:C} is an immediate consequence of the combination of Proposition \ref{prop:p-to-a}, Theorem \ref{ithr:D} and a classical removability result \cite[Proposition 2.7.13]{Fresnel04}.
  \begin{proof}[Proof of Theorem \ref{ithr:C}]
    The first assertion follows immediately from Proposition \ref{prop:p-to-a} and Theorem \ref{ithr:D}.  If $\cJ(\lambda_0) \subset \aone$,
    then $\cJ(\lambda_0)$ is analytically removable, by Theorem~\ref{thr:removability} (see \cite[Proposition 2.7.13]{Fresnel04}). Thus any analytic conjugacy between $f_{\lambda_1}:\cB(\lambda_1) \to \cB(\lambda_1)$ and $f_{\lambda_2}:\cB(\lambda_2) \to \cB(\lambda_2)$ extends to an analytic automorphism of $\aoneberk$.  Therefore $f_{\lambda_0}$ and $f_\lambda$ are affinely conjugate for all $\lambda\in\Lambda$. Since there are finitely many elements in $\Poly_d$ that are affinely conjugate to $f_{\lambda_0}$, we conclude that $f_\lambda=f_{\lambda_0}$ for all $\lambda\in\Lambda$. Thus the  second assertion holds.
 \end{proof}

\section{Polynomials with Julia critical points}
\label{sec:rigidity}

The aim of this section is to prove Theorem~\ref{ithr:A} and Corollary~\ref{icor:B}.

\subsection{Bounded Fatou components of tame polynomials}
\label{sec:bounded}

 For a point $x\in\aoneberk$ and a polynomial $f$, denote by $\omega(x)$ the $\omega$-limit set of $x$ under $f$.
 The following result, due to Trucco, shows that
 the orbit of $x \in \cJ(f) \cap \H$
 accumulates inside a hyperbolic ball around the base point $x_f$ of radius proportional
 to $\dist(x,x_f)$.

\begin{proposition}[{\cite[Proposition 7.1]{Trucco14}}]\label{p:distance}
  Suppose that $f$ is a tame polynomial of degree at least $2$.
  If $x\in\cJ(f) \cap \H$, then for all $y\in\omega(x)$
  we have that $y \in \H$ and
$$\dist (y,x_f)\le d^{d-1}\dist (x,x_f).$$
\end{proposition}

%The following result is a consequence of results in \cite{Trucco14}, which is a  prerequisite for one direction of Theorem \ref{ithr:A}.

One direction of Theorem \ref{ithr:A} is a consequence of the following:

% \begin{proposition}\label{prop:diameter}
%   Let $f$ be a tame polynomial and 
% suppose that $\cJ(f)\cap\H\not=\emptyset$. Then there exist a Fatou component $\Omega$ with $\Omega\not=\cB(f)$ and $\epsilon>0$ such that for all $j\ge 0$, the diameter of $f^{\circ j}(\Omega)$ is at least $\epsilon$.
% \end{proposition}
% \begin{proof}
% If $\cJ(f)\cap\H$ contains a preperiodic point, the conclusion follows immediately. Now we assume that all points $\cJ(f)\cap\H$ are wandering and pick a point $x\in\cJ(f)\cap\H$. By \cite[Corollary 7.5]{Trucco14}, the set $\omega(x)$ contains a recurrent ramified point $y \in \H$, that is $y\in\omega(x)\cap\cR(f)$ and $y\in\omega(y)$. It follows that  $y\in\omega(f^{\circ  j}(y))$ for all $j\ge 0$. By Proposition~\ref{p:distance}, we have for all $j\ge 0$,
% $$d_\H(y,x_f)\le d^{d-1}d_\H(f^{\circ j}(y),x_f).$$
% It follows that there exists $\epsilon>0$ such that $\diam(f^{\circ j}(y))\ge\epsilon$ for all $j\ge 0$. Considering the Fatou component $\Omega$ with boundary $\partial\Omega=\{y\}$, we obtain the conclusion.
% \end{proof}

\begin{corollary}\label{coro:oneA}
  If $f \in \Poly_d$ is a tame polynomial in the closure of the tame shift locus, then $\cJ(f) \subset \aone$. 
\end{corollary}

\begin{proof}
  %We first show the implication that if $f$ is contained in the boundary of tame shift locus, then $\cJ(f)\subset\aone$. 
  Suppose on the contrary that $\cJ(f)\cap\H\not=\emptyset$. {Pick $x_0\in\cJ(f)\cap\H$.}
  By Proposition \ref{p:distance}, for all  $y \in \omega(x_0)$
  we have that $\dist (f^{\circ n}(y), x_f) \le R$ for some $R >0$.
  Let $\cD$ be a neighborhood of $f$ in $\Poly_d$ such that for all $g \in \cD$
  we have that $f(x)=g(x)$ for all $x$ such that $\dist (x, x_f)\le R$.
  It follows that every $y \in \omega(x_0) \subset \H$ lies in the filled Julia set $\cK(g)$
  for all $g \in \cD$. Therefore, $\cD$ is free of polynomials in the tame shift locus.
\end{proof}

A necessary and sufficient condition for the abscence of bounded Fatou
  components is provided by our following result:
%Next we show that the Julia set is formed by type I points if there is no critical point in the bounded Fatou components:
\begin{proposition}\label{p:Julia type I}
Let $f\in\Polyd$ be a tame polynomial. Then 
  $\crit(f)\subset\cB(f)\cup\cJ(f)$ if and only if $\cJ(f)\subset\aone$.
\end{proposition}
\begin{proof}
  If $\cJ(f) \subset \aone$, then $\aoneberk = \cB(f)\cup\cJ(f)$ clearly contains all the critical points of $f$.
  We assume that $\crit(f)\subset\cB(f)\cup\cJ(f)$ and proceed by contradiction to show that $\cJ(f)\subset\aone$.
  Hence, suppose  that there exists $x\in\cJ(f)\setminus\aone$. For each $j\ge 0$, set $x_j:=f^{\circ j}(x)$ and regard
  $]x_j,x_f[$ as  an open subtree equipped with the vertices  $\{x_j^{(n)}\}_{n\ge 1}$ formed by all iterated preimages of $x_f$ in
  $]x_j,x_f[$ in decreasing order (i.e. $ x_j^{(n)} \succ x_j^{(n+1)}$ for all $n$).
  Then $f^{\circ n} (x_j^{(n)}) = x_f$ and  $x_j^{(n)}\to x_j$ as $n\to\infty$, see \cite[Proposition 3.6]{Trucco14}.
  We claim that there exists a uniform $N\ge 1$ such that the degree of $f$ on $]x_j, x_j^{(n)}[$ is $1$ for all $j\ge 0$ and $n\ge N$. For otherwise, there would exist a critical point $c$, $n_k \to \infty$ and $j_k \ge 0$
  such that $\{x_{j_k}^{(n_k)}\}_k \in ]c, x_f]$ is a decreasing sequence. Then both $x_{j_k}$ and $x_{j_k}^{(n_k)}$ would converge, as $k \to \infty$,
  to some $y \succeq c$. By
  Proposition \ref{p:distance},  we would have that $y \in \cJ(f)\cap \H $ and therefore $\ol{D}_y \subset \cK(f)$, which
  contradicts that $\crit(f)\subset\cB(f)\cup\cJ(f)$ since $\ol{D}_y\setminus\{y\}$ is contained in $\cF(f)\setminus\cB(f)$.

  Now from Lemma \ref{lem:factor}, it follows that $f^{\circ(\ell+1)}$ maps $]x_0^{(N+\ell+1)},x_0^{(N+\ell)}[$ isometrically onto $]x_{\ell+1}^{(N)},x_{\ell+1}^{(N-1)}[$ for any $\ell\ge 0$. Morover, there are only finitely many arcs of the form $]x_{\ell+1}^{(N)},x_{\ell+1}^{(N-1)}[$. Therefore, the
    length of $]x_0^{(N+\ell+1)},x_0^{(N+\ell)}[$ is uniformly bounded below for all $\ell$. We conclude that
 $$\dist(x,x_f)\ge\sum_{\ell\ge 0}\dist(x_0^{(N+\ell+1)},x_0^{(N+\ell)})=\infty,$$  
 and hence $x\in\aone$, which contradicts the assumption that $x\in\cJ(f)\setminus\aone$. 
\end{proof}

\subsection{Moving critical points to the basin of infinity}
\label{sec:perturbation}

The following result establishes a strong form of density of polynomials
with all their critical points passive. To state the result let us agree
that if $\Lambda \subset \field$ is a disk and $\lambda_0 \in \Lambda$,
a \emph{maximal disk $\Lambda'$ in $\Lambda \setminus \{\lambda_0\}$} is
an open disk of radius $s$ such that $s = |\lambda' - \lambda_0|$ for all $\lambda' \in \Lambda'$.

\begin{lemma}\label{lem:small-disk}
  Let $\Lambda$ be a disk of radius $r>0$ parametrizing a critically
  marked analytic family $\{f_\lambda\}$ of tame polynomials in
  $\Poly_d$ with the same base point $x_\Lambda$. Given $\lambda_0 \in \Lambda$,  {for any $s <r$ sufficiently close to $r$,} 
  there exist a maximal disk $\Lambda'\subset\Lambda \setminus \{\lambda_0\}$ of radius $s$ such that the following hold:
  \begin{enumerate}
  \item all the critical points of $\{f_\lambda\}$ are passive in $\Lambda'$; and
  \item if $c_i(\lambda)$ is active in $\Lambda$, then $c_i(\lambda') \in \cB(\lambda')$ for all $\lambda' \in \Lambda'$.
  \end{enumerate}
\end{lemma}

\begin{proof}
  Let $c_1(\lambda), \dots, c_k(\lambda)$ be the active critical points
  in $\Lambda$. For all $j$ we have let $g_{j,n} (\lambda):=f^{\circ n}_\lambda (c_j(\lambda))$. Note that $\sup \{ |g_{j,n}(\lambda)| : \lambda \in \Lambda \}$
  converges to $\infty$ as $n \to \infty$. Thus, there
  exists $N_j$ and $s_j<r$ such that $|x_\Lambda| < r_{j,n} (t):=\sup \{ |g_{j,n}(\lambda)| : |\lambda-\lambda_0| \le t \}$
  for all $s_j < t < r$ and $n \ge N_j$. Pick $s<r$ arbitrarily close to $r$ 
  in the value group so that $s>s_j$ for all $j=1, \dots, k$.
  Then $g_{j,N_j}$ maps all but finitely maximal open disks contained in $|\lambda-\lambda_0| =s$  onto a maximal open disk contained in the ``sphere''
  $S_j:=\{z \in \aone: |z|=r_{j,N_j} (s) \}$.
  Since $S_j \subset \cB(\lambda)$ for all $\lambda \in \Lambda$,
  we may choose a maximal open disks $\Lambda'$
  contained in $|\lambda-\lambda_0| =s$ such that for all $j=1, \dots, k$, we have $g_{j,N_j}(\Lambda') \subset S_j$ 
  and the conclusions of the lemma hold.   
\end{proof}

\begin{corollary}\label{coro:passive}
  Suppose $\{f_\lambda\}$ is a critically marked analytic family of tame polynomials in $\Poly_d$
  parametrized by a disk $\Lambda$ with the same base point $x_\Lambda$. 
  Assume that $c_1(\lambda), \dots, c_k(\lambda)$ are passive critical points in $\Lambda$
  and that  $c_{k+1}(\lambda), \dots, c_{d-1}(\lambda)$ are active critical points in $\Lambda$.
  Given $\lambda_0 \in \Lambda$ there exists {a disk} $\Lambda' \subset \Lambda$. such that the following hold:
  \begin{enumerate}
  \item all the critical points are passive in $\Lambda'$;
  \item for all $k+1\le j\le d-1$ and all $\lambda' \in \Lambda'$,
    we have $c_j(\lambda') \in \cB(\lambda')$; and 
  \item for all $1\le j \le k$, if $c_j(\lambda_0) \in \cJ(\lambda_0)$, then
    $c_j(\lambda') \in \cJ(\lambda')$ for all $\lambda' \in \Lambda'$.
  \end{enumerate}
\end{corollary}

\begin{proof}
  Given $\lambda_0$, consider a disk
  $\Lambda'$ furnished by Lemma \ref{lem:small-disk}. Hence (1) and (2) hold for $\Lambda'$. Assume that $c_j(\lambda_0) \in \cJ(\lambda_0)$ and
  $c_j(\lambda)$ is passive in $\Lambda$. Then $c_j(\lambda) \in \cK(\lambda)$
  for all $\lambda \in \Lambda$. Denote by $\delta(\lambda) \ge 0$ the diameter
  of the component of $\cK(\lambda)$ containing $c_j(\lambda)$.
  To establish (3) for $c_j$ we must show that $\delta(\lambda)=0$ for all $\lambda \in \Lambda'$.  We proceed by contradiction.
  Suppose that there exists $\lambda' \in \Lambda'$ so
  that $\delta(\lambda') >0$. Then {by Propositions \ref{prop:close} and \ref{prop:p-to-a}} for every closed disk $\ol{\Lambda''} \subset \Lambda'$ {with $\lambda'\in\Lambda''$}, there exists $\rho > 0$ such that the dynamics of $f_\lambda$ over the corresponding $\rho$-trimmed dynamical core $\cT_\lambda$ is isometrically conjugate to the action of $f_{\lambda'}$ on $\cT_{\lambda'}$ for all $\lambda\in\Lambda''$. In particular, the iterated preimages of $x_\Lambda$ in $]c_j(\lambda), x_\Lambda]$ isometrically correspond
  to those of $x_\Lambda$ in $]c_j(\lambda'), x_\Lambda]$. It follows that
  $\delta(\lambda) = \delta(\lambda')$ for all $\lambda \in \overline{\Lambda''}$ for every closed disk $\overline{\Lambda''}$ contained in $\Lambda'$.
  Therefore, $\delta(\lambda) = \delta(\lambda')$ for all $\lambda \in \Lambda'$.
  Modulo change of coordinates, we may assume that $c_j(\lambda) =0$ for all $\lambda \in \Lambda$. Then, given $y$ with $|y| \le \delta(\lambda')$ and $n \ge 1$,
  we have $|f^{\circ n}_\lambda (y)| \le |x_\Lambda|$ for all $\lambda \in \Lambda'$. Since $\Lambda'$ is a maximal open disk in $B=\{\lambda : |\lambda-\lambda_0| \le s \}$ and $\lambda \mapsto f^{\circ n}_\lambda (y)$ is analytic on $B$, we have $|f^{\circ n}_\lambda (y)| \le |x_\Lambda|$ for all $\lambda \in B$.
In particular, since $\lambda_0 \in B$, we have that $\delta(\lambda_0) \ge \delta(\lambda')$ which contradicts our assumption that $c_j(\lambda_0) \in \cJ(\lambda_0)$.   
\end{proof}

{Recall that the space of polynomials $\Poly(\mathbf{d})$ with constant critical multiplicity was introduced  in Section \ref{sec:poly}. Given $f \in \Poly(\mathbf{d})$ with at least one non-escaping critical point, the following result guarantees the existence of a one parameter analytic family passing through $f$ with certain properties that will be need in the proofs of Theorem \ref{ithr:A} and Corollary \ref{icor:B}.}

\begin{lemma}
  \label{l:analytic}
  Assume that $(f, c_1, \dots, c_k)$ is a nonsimple tame polynomial in
  $\Poly(\mathbf{d})$  such that $c_i \in \cB(f)$ for $1 \le i \le j$.
  If $j <k$, then there exists a nonconstant
  analytic family of tame polynomials
  $\{(f_\lambda,c_1(\lambda), \dots,c_k(\lambda))\}$ parametrized by a disk
  $\Lambda \subset \field$ with $0 \in \Lambda$ such that
  $(f,c_1,\dots,c_k)=(f_0,c_1(0), \dots,c_k(0))$, the base point of $f_\lambda$ is independent of $\lambda \in \Lambda$ and the 
  B\"ottcher coordinates of $c_1(\lambda), \dots, c_j(\lambda)$ are constant.
\end{lemma}

\begin{proof}
  For each $i \le j$,
  let $n_i$ be such that $|f^{\circ n_i}(c_i)| > |x_f|$.
  Recall that $\Poly(\mathbf{d})$ is 
  naturally identified with elements of a hyperplane in $\field^k \times \field$ which are
  in the complement
  of a finite collection of linear subspaces of $\field^k \times \field$ and
  that tame polynomials are open in $\Poly(\mathbf{d})$ (endowed with the sup-norm). Thus we may assume that there exists a polydisk $P \subset
  \Poly(\mathbf{d})$ of tame polynomials
  containing $f$. Shrinking $P$ if necessary, we may also assume that for all $g \in P$
  the basepoint of
  $g$ is $x_f$  and that for every $i \le j$ we have $|g^{\circ n_i}(c_i(g))| > |x_f|$. Without loss of generality we
  identify $P$ with $\ol{\D}^k$ where $\ol{\D} = \{ z \in \field : |z| \le 1 \}$.

% {\tiny The proof relies on some results about Tate algebras (e.g. see~\cite{Fresnel04}).
%   For $n \ge 1$, consider 
%   $\br=(r_1,\dots,r_n)$ where $r_i \in |\field^\times|$ for all $i$.
%   Denote by $P(\br)$ the polydisk in $\field^n$ formed by all
%   $(\lambda_1, \dots, \lambda_n) \in \field^n$ such that $|\lambda_i| \le r_i$ for all $i$. Then the \emph{Tate algebra} $T_n[\br]$ is
%   the one formed by all (formal) power series in $\field \llbracket \lambda_1, \dots, \lambda_n \rrbracket$ convergent in $P(\br)$. }
  
    For $n \ge 1$, consider 
    $\br=(r_1,\dots,r_n)$ where $r_i \in |\field^\times|$ for all $i$. Recall
    that $P(\br)$ denotes the corresponding polydisk and  $T_n[\br]$ is the associated Tate algebra, as in the end of Section \ref{sec:poly}.
If $r_i =1$ for all $i$, we simply write $T_n$ for the Tate algebra and $\ol{\D}^n$ for the corresponding polydisk. 
  Note that the maximal spectrum $\operatorname{Sp}(T_n)$ is in natural bijection with $\ol{\D}^n$. 

By Lemma \ref{lem:B-analytic},
  we have that  $F_i(g):=1/\phi_g (g^{\circ n_i}(c_i(g)))$ lies in the Tate algebra $T_k$ for all $i\le j$.
  Let $I$ be the ideal generated by $F_1(g)-F_1(f), \dots,
  F_j(g)-F_j(f)$
  in $T_k$, which we may assume to be a radical ideal.
  Our aim is to find the analytic image of one-dimensional disk contained in the vanishing locus $V \subset \ol{\D}^k$ of $I$.
  Since $j <k$, it follows that $T_k/I$ is a reduced affinoid algebra of dimension at least $1$. Enlarge $I$, if necessary, to obtain an ideal $J$ such that $A=T_k/J$ is reduced and has dimension~$1$.
  Let $B$ be a normalisation of $A$. That is, $B$ is an integrally closed affinoid algebra of dimension~$1$ and $A \hookrightarrow B$ is a finite morphism. 
Take a maximal ideal $\mathfrak{m}_f$ in $B$ which maps onto $f$ via
  $\operatorname{Sp}(B) \to \operatorname{Sp}(A) \hookrightarrow \operatorname{Sp}(T_k) \to \ol{\D}^k$. 
  Since the dimension of $B$ is $1$, the localization $B_{\mathfrak{m}_f}$ is regular (e.g. see~\cite[Proposition 9.2]{Atiyah69}). By~\cite[Theorem~3.6.3]{Fresnel04}, $B$ can be represented by
  $T_n/(G_1,\dots,G_s)$ such that a  $(n-1)\times(n-1)$-minor of the
  Jacobian matrix
  $(\partial G_t/\partial \lambda_m)$, modulo  $(G_1,\dots,G_s)$, is not in $\mathfrak{m}_f$. To fix ideas, let us assume that $\mathfrak{m}_f =(\lambda_1, \dots, \lambda_n)$ (i.e. it corresponds to the origin in $\ol{\D}^n$) and let $J':=(G_1,\dots,G_s)$.
  By the Implicit Function Theorem (e.g. see~\cite[II.III.10]{Serre64}),
  after relabeling if necessary, there exists $\epsilon_i >0$ in $|\field^\times|$
  such that  $B_\epsilon = T_n[(\epsilon_i)]/J'$  is isomorphic to $T_1[\epsilon_1]$. Thus,
    $T_k \to A \to B \to B_\epsilon \to T_1[\epsilon_1]$ gives a nonconstant analytic map $\lambda \mapsto f_\lambda$
  from $\{ \lambda  : |\lambda | \le \epsilon_1 \}$ into $\ol{\D}^k$
  such that $f_0 = f$ and $f_\lambda \in V$ for all $\lambda$
  (i.e. the escaping critical points $c_1(\lambda), \dots, c_j(\lambda)$ of $f_\lambda$ have constant B\"ottcher coordinates).
\end{proof}

Now we can prove Theorem \ref{ithr:A} and Corollary \ref{icor:B}.
\begin{proof}[Proof of Theorem \ref{ithr:A}]

{Since we have already established Corollary \ref{coro:oneA},}
now assume that $f \in \Poly_d$ is a tame polynomial  such that
$\cJ(f) \subset \aone$ and $\crit(f) \cap J(f) \neq \emptyset$.
To prove that $f$ is in the closure of the tame shift locus,
it will be sufficient to show that there exists an arbitrarily close
tame polynomial $g \in \Poly_d$ such that $\cJ(g) \subset \aone$ and
$\# \crit(g) \cap J(g) < \# \crit(f) \cap J(f)$.

By Lemma \ref{l:analytic} we can consider a nonconstant one-dimensional critically marked analytic family $\{f_\lambda\}_{\lambda\in\Lambda}$ in $\Poly_d$
parametrized
by an open disk with $0\in\Lambda \subset \field$ such that $f_0=f$, 
the base point $x_\Lambda$ is of $f_\lambda$ is independent of $\lambda\in\Lambda$ and the B\"ottcher coordinates of escaping critical points
of $f_0$ are constant.
Since $\cJ(f_0)\subset\aone$, we conclude that at least one critical point of $f_\lambda$ is active in $\Lambda$; for otherwise, by Theorem \ref{ithr:C}, the map $f_\lambda$ would be  affine conjugate to $f_0$ for all $\lambda\in\Lambda$, which is a contradiction since affine conjugacy classes in $\Poly_d$ are finite.
By Corollary \ref{coro:passive}, there exists $\lambda' \in \Lambda$ such that
all the critical points of $f_{\lambda'}$ are in $\cB(f_{\lambda'})$ or in $\cJ(f_{\lambda'})$ and the number of Julia critical points of  $f_{\lambda'}$
is strictly smaller than the number of Julia critical points of  $f_{0}$.
It follows that $g=f_{\lambda'}$ is such that $\cJ(g) \subset \aone$ and
$\#\left(\crit(g) \cap\cJ(g)\right) < \#\left(\crit(f) \cap\cJ(f)\right)$.
\end{proof}

\begin{proof}[Proof of Corollary \ref{icor:B}]
Suppose on the contrary that $c_i$ is passive. Corollary \ref{coro:passive} implies that the critical point $c_i(g)$ is contained in $\cJ(g)$ for all $g$ in a sufficiently small neighborhood of $f$, which contradicts Theorem \ref{ithr:A}.
\end{proof}

\appendix
\section{A local lemma}
\label{appendix}
For the sake of completeness we present here the proof of a result which easily follows along the lines of Hensel's Lemma.

\begin{lemma} \label{lem:lift}
Consider two polynomials  $f,g\in\field[z]$  and
%Let $f\in\mathcal{O}_K[z]$ be a tame polynomial of degree at least $2$ and let $g\in\mathcal{O}(U)$ be an analytic function in an affinoid $U$ with $\xi_G\in U$. 
suppose that 
\begin{enumerate}
\item $f(x_G)=x_G=g(x_G)$,
%\item $(T_{\xi_G}f)^{-1}(\overrightarrow{\xi_G\infty})=\{\overrightarrow{\xi_G\infty}\}$,
%\item each $f$ and $g$ has finitely many bad directions in $T_{\xi_G}\mathbf{P}^1$,
\item $T_{x_G}f=T_{x_G}g$, and
\item $f'(x_G)=x_G$. 
\end{enumerate}
Then there exist  an affinoid $V$ containing $x_G$ and an analytic function $h: V\to h(V)$ such that
\begin{enumerate}[label=(\alph*)]
%\item $h(\xi_G)=\xi_G$,
\item $f\circ h=g$ in $V$, and
%\item $T_{\xi}h=id$ for $\xi\in V\cap[\xi_G,\infty)$.
%\item $h(\xi)=\xi$ and $T_{\xi}h=id$  for $\xi\in V\cap[\xi_G,\infty)$.
\item $h(x_G)=x_G$ and $T_{x_G}h= \operatorname{id}$.
\end{enumerate}
\end{lemma}

\begin{proof}
   Denote by $d \ge 2$ the degree of $f$. 
  For $\rho >1$ sufficiently close to $1$, there exists $0 < s <1$ such that
  $|f(z) - g(z)| \le s$ for all $z \in \aone$ with $|z| \le \rho$.
  Choose $s < \mu < 1$ and consider $1 < r < \rho$ such that:
  $$\mu r^d < 1 \text{ and } sr^d < \mu.$$
  Let $V$ be a closed (rational) affinoid contained in $D(0,r)$ and containing $x_G$ in its interior such that  $\diam(x) > \mu$ for all $x \in \partial V$ and the following hold for all
  $z \in V \cap \aone$:
  $$\mu r^d < |f'(z)| \text{ and } sr^d < \mu \cdot |f'(z)|^2.$$

  Observe that if $z \in V \cap \aone$ and $w \in \aone$ is such that $|w| \le \mu$ then $z+w \in V$ and, moreover,
  $$|f'(z +w)| =  |f'(z)|;$$
  indeed, for some polynomials $a_j(z)$ of degree at most $d$ with coefficients in the closed unit ball $\mathfrak{O} \subset \field$,
  $$|f'(z+w) -f'(z)| = |\sum_{j \ge 1} a_j(z) w^j| \le r^d |w| < |f'(z)|.$$

  Given $x \in V \cap \aone$, for $n\ge 0$, let
  \begin{eqnarray*}
    z_0 (x) &= & x,\\
                 w_n(x)  &=& - \dfrac{f(z_n(x))-g(x)}{f'(z_n(x))},\\
                 z_{n+1}(x) & = & z_n(x) + w_n(x).
  \end{eqnarray*}
  Note that $|w_0(x)| \le s/|f'(x)| < s \sqrt{\mu/sr^d} < \mu$.
 We claim that for all $n \ge 0$, 
  \begin{eqnarray*}
    |f(z_{n+1}(x))-g(x)| & < & \mu |f(z_{n}(x))-g(x)|,\\
    |w_{n+1}(x)| &<& \mu |w_n(x)|.
  \end{eqnarray*}
Let us proceed by induction. We will omit the case $n=0$ for the first inequality since the inductive step is very similar, so consider $n \ge 1$ and
  assume that the inequalities hold for all $k < n$. Observe that
  $$f(z_n(x) + w_n(x))- g(x) = \sum_{j\ge 2} f_j(z_n(x)) w_n(x)^j$$
  for some $f_j(z) \in \mathfrak{O}[z]$ of degrees at most $d$.
  Hence
  \begin{eqnarray*}
    |f(z_{n+1}(x)) - g(x)| &\le& r^d |w_n(x)|^2 = r^d \dfrac{|f(z_n(x))-g(x)|^2}{
                                 |f'(z_n(x))|^2}\\
                           & \le & r^d |f(z_n(x))-g(x)| \dfrac{s}{|f'(z_n(x))|^2}\\
    & < & \mu |f(z_n(x))-g(x)|,
  \end{eqnarray*}
  where the last inequality follows from $|f'(x)| = |f'(z_n(x))|$ since
  $|w_k(x)| < \mu$ for all $k < n$. 
Similarly, {noting that  $|w_n(x)| < \mu$ and hence $|f'(z_{n+1}(x))|= |f'(z_n(x))|$, we have}
\begin{eqnarray*}
  |w_{n+1}(x)| &\le& \dfrac{r^d |w_n(x)|^2}{|f'(z_n(x))|} \\
  & = &  |w_n(x)| \dfrac{r^d |f(z_n(x))-g(x)| }{|f'(z_n(x))|^2}\\
&\le& |w_n(x)| \dfrac{r^d s}{|f'(z_n(x))|^2} < \mu |w_n(x)|  .
\end{eqnarray*}

To finish the proof observe that $w_n(x)$ converges to $0$ uniformly in $V$  as $n\to \infty$, and hence $z_n(x)$ converges to an analytic function $z(x)$ in $V$. Moreover,
$$f(z(x))-g(x)=\lim_{n\to\infty}f(z_n(x))-g(x)=0.$$
Also  $|z(x)-x|<\mu$ since $|z_n(x)-x|<\mu$ for all $n$.
The conclusion of the lemma follows immediately by letting $h(x)=z(x)$.

\end{proof}

\bibliographystyle{alpha}
\bibliography{references}

\end{document}